\documentclass[11pt,a4paper]{article}
\usepackage[utf8]{inputenc} 
\usepackage[T1]{fontenc} 
\usepackage[german,english]{babel}
\usepackage[style=alphabetic,backend=biber]{biblatex}
\bibliography{regerror}
\usepackage{mathtools, amssymb, amsthm} 
\DeclareMathOperator*{\diam}{diam}
\DeclareMathOperator*{\identity}{Id}

\DeclareMathOperator*{\meas}{meas}
\DeclareMathOperator*{\range}{range}
\DeclareMathOperator*{\argmin}{argmin}
\usepackage{enumerate} 
\usepackage{graphicx}
\usepackage{subfig}
\usepackage[colorlinks=false]{hyperref}
\usepackage{colortbl}

\title{Tikhonov regularization of 
         control-constrained optimal control problems}
\author{
Nikolaus von Daniels
\footnote{Schwerpunkt Optimierung und Approximation,
Universität Hamburg, Bundesstraße~55, 20146~Hamburg, Germany,
\texttt{nvdmath@gmx.net}}
}

\date{December 7, 2017}

\begin{document}

\maketitle

\newtheorem{theo}{Theorem}
\newtheorem{algo}[theo]{Algorithm}
\newtheorem{defi}[theo]{Definition}
\newtheorem{lemm}[theo]{Lemma}
\newtheorem{coro}[theo]{Corollary}
\newtheorem{assu}[theo]{Assumption}
\newtheorem{rema}[theo]{Remark}
\newtheorem{exam}[theo]{Example}

\DeclarePairedDelimiter\abs{\lvert}{\rvert}
\DeclarePairedDelimiter\norm{\lVert}{\rVert}
\newcommand{\twoset}[2]{\ensuremath{\left\{#1\,\left|\;#2\right.\right\}}}
\newcommand{\dual}[1]{\ensuremath{{#1}^*}}
\newcommand{\Uad}{U_\textup{ad}}
\newcommand{\uopt}{\ensuremath{\bar u}}
\newcommand{\yopt}{\ensuremath{\bar y}}
\newcommand{\popt}{\ensuremath{\bar p}}
\newcommand{\uoptd}{\ensuremath{\bar u_{kh}}}
\newcommand{\yoptd}{\ensuremath{\bar y_{kh}}}
\newcommand{\poptd}{\ensuremath{\bar p_{kh}}}
\newcommand{\V}{\ensuremath{H^1_0(\Omega)}}
\newcommand{\Vd}{\ensuremath{H^{-1}(\Omega)}}
\renewcommand{\H}{\ensuremath{L^2(\Omega)}}  
\newcommand{\Loo}{\ensuremath{L^\infty(\Omega)}} 
\newcommand{\HI}{\ensuremath{H^1(\Omega)}} 
\newcommand{\LIIH}{\ensuremath{L^2(I,\H)}}
\newcommand{\LIIV}{\ensuremath{L^2(I,\V)}}
\newcommand{\LIIVd}{\ensuremath{L^2(I,\Vd)}}
\newcommand{\LooLoo}{\ensuremath{\infty}}
\newcommand{\nvdrem}[1]{\textcolor{red}{\tiny #1}}
\newcommand{\person}[1]{\textsc{#1}}
\newcommand{\restr}[2]{{
  \left.\kern-\nulldelimiterspace 
  #1 
  \vphantom{\big|} 
  \right|_{#2} 
}}

\noindent {\small {\bf Abstract:}
We consider Tikhonov regularization of control-constrained optimal
control problems.
We present new a-priori estimates for the regularization 
error assuming measure and source-measure conditions.
In the special case of bang-bang solutions, 
we introduce another assumption to
obtain the same convergence rates. This new condition 
turns out to be useful in the derivation of error estimates for
the discretized problem. 
The necessity of the just mentioned assumptions
to obtain certain convergence rates is analyzed.
Finally, a numerical example confirms the analytical findings.
}\\[2mm]


\section{Introduction}

In this article we study the regularization of the
minimization problem
\begin{equation}\label{OCPl}\tag{$\mathbb P_0$}
   \min_{u\in\Uad} J_0(u)\quad\text{with}\quad
          J_0(u):= \frac{1}{2} \norm{Tu-z}^2_{H}
\end{equation}
for $T:U \to H$ a given linear and continuous operator between
the control space  $U:=L^2(\Omega_U)$ 
with scalar product $(\cdot,\cdot)_U$
and an arbitrary Hilbert space $H$ 
where $z\in H$ is a fixed function to be approached.
The set $\Omega_U\subset\mathbb R^n$, $n\ge 1$, is a bounded 
measurable domain and 
the set of admissible controls $\Uad\subset U$ is given by
\begin{equation} \label{E:Uad}
   \Uad:=\twoset{u\in U}{
     a(x)\le u(x)\le b(x)\quad \text{for almost all $x\in\Omega_U$}} 
\end{equation} 
with fixed control bounds $a$, $b\in L^\infty(\Omega_U)$ 
fulfilling $a\le b$.

We give two instances of $T$ as solution operators
of linear partial differential equations (PDEs): 
\begin{exam}\label{exam:poisson}
Let $y$ be the unique weak solution of the Poisson problem
\begin{equation}\label{e:poisson}
\begin{aligned}
 -\Delta y & = u     &&\text{in $\Omega$,}\\
         y & = 0     &&\text{on $\partial\Omega$}
\end{aligned}
\end{equation}
for given $u\in\H$ on some bounded sufficiently regular domain
$\Omega\subset \mathbb R^d$, $d\ge 1$, with boundary $\partial\Omega$.
 
We set $\Omega_U:=\Omega$ and
get $y=Tu$ where $T:U=\H \to H:=\H$ is the weak solution operator 
associated with problem \eqref{e:poisson}.
\end{exam}
\begin{exam}\label{exam:timedep}
  Consider the heat equation
  \begin{equation}\label{e:heateq}
      \begin{aligned} 
        \partial_t y -\Delta y &= Bu &&\text{in }I\times\Omega\,,\\
         y&=0&&\text{in } I\times\partial\Omega\,,\\
          y(0)&=0&&\text{in } \Omega\,.
  \end{aligned}
  \end{equation}
  with a control operator $B:U\to \LIIVd$.

  We fix a time interval $I:=(0,T_e)\subset \mathbb R$
  with a given end-time fulfilling $0<T_e<\infty$.
  Furthermore, we assume $\Omega$ to be a domain as in the previous 
  example.
  Let $T:=SB$ be the control-to-state operator with
  $S:L^2(I,\Vd)\allowbreak \to H:=\LIIH$
  being the weak solution operator for the heat equation
  \eqref{e:heateq}. We will discuss it later from \eqref{E:operatorS}
  onwards.

  Let us mention two instances for the control operator B:
  \begin{enumerate}
  \item (Distributed controls) 
  We set $\Omega_U := I\times\Omega$.
  The control operator $B:U=H\rightarrow \LIIVd$ is given by 
  $B:=\identity$, i.e., the identity mapping induced by the
  standard Sobolev imbedding $\iota: \H\hookrightarrow\Vd$. 
  
  \item (Located controls) 
  Let $\Omega_U := I$ and $g_1\in \H$ be a fixed function.
  The operator $B$ given by
  \begin{equation}\label{E:Bloccont}
  B: L^2(I,\mathbb R)\rightarrow \LIIVd\,,\quad u\mapsto 
        \left( t\mapsto u(t)\iota g_1 \right),
  \end{equation}
  with $\iota$ from the previous item maps a control function
  $u$ depending only on time to a function distributed in space-time.

  With little more effort one can consider the case of several fixed
  functions $g_1, \dots, g_D$, replacing $B$ by
  $u\mapsto \left( t\mapsto \sum_{i=1}^D u_i(t)\iota g_i \right)$
  and seeking for control functions $u_1, \dots, u_D$.
  We omit this generalization here to shorten the exposition and refer
  the interested reader to \cite{dissnvd}.
  \end{enumerate}
\end{exam}
%
To unify the examples just given, it is useful to write $T=SB$ with two
continuous linear operators $B:U\to R$ and $S:R\to H$ where $R$ is an 
appropriately chosen function space.  This decomposition is always 
possible for a given $T$ by taking $B=\identity$ and $S=T$ (or vice versa).

Often, the solutions of \eqref{OCPl} possess a special structure:
They take values only on the bounds $a$ and $b$ of 
the admissible set $\Uad$ given in \eqref{E:Uad}
and are therefore called \emph{bang-bang solutions}.

Theoretical and numerical questions related to this control problem
attracted much interest in recent years, see, e.g., 
\cite{deckelnick-hinze},
\cite{wachsmuth1},
\cite{gong-yan},
\cite{wachsmuth2},
\cite{wachsmuth3},
\cite{wachsmuth4},
\cite{wachsmuth5},
\cite{felgenhauer2003},
\cite{alt-bayer-etal2},
\cite{alt-seydenschwanz-reg1},
and
\cite{seydenschwanz-regkappa}. 
The last four papers are concerned with $T$ being the solution operator
of an \emph{ordinary} differential equation, the first three papers with $T$ 
being a solution operator of an \emph{elliptic} PDE 
as in Example~\ref{exam:poisson}, and the remaining references 
with $T$ being a general linear operator as here.
In \cite{dissnvd}, a brief survey of the content of these and some
other related papers is given at the end of the bibliography.
For an appropriate discretization of Example~\ref{exam:timedep} 
we refer to \cite{DanielsHinzeVierling}, \cite{dissnvd},
and the forthcoming \cite{danielshinze}, but see also the numerics section
below.

Problem~\eqref{OCPl} is in general ill-posed, meaning that 
a solution does not depend continuously on the datum $z$,
see \cite[p. 1130]{wachsmuth2}.
The numerical treatment of a discretized problem version
is often challenging or even impossible.

Therefore, we use Tikhonov regularization to overcome
these difficulties. The \emph{regularized problem} is given by
\begin{equation}\label{OCP}\tag{$\mathbb P_\alpha$}
   \min_{u\in\Uad} J_\alpha(u)\quad\text{with}\quad
          J_\alpha(u):= \frac{1}{2} \norm{Tu-z}^2_{H}
                              + \frac{\alpha}{2} \norm{u}^2_{U}
\end{equation}
where $\alpha > 0$ denotes the regularization parameter.

Formally, for $\alpha=0$ problem \eqref{OCP} reduces to problem
\eqref{OCPl} also called the \emph{limit problem}.

Note that for the regularized problems \eqref{OCP}, $\alpha > 0$,
and their discretizations, explicit solution representations are 
available and can be utilized for numerical implementation; cf. 
\eqref{FONC}, \eqref{FONCkh} below.

We recall in the \textbf{next section}
basic properties of the regularized and the limit problem: 
Problem \eqref{OCP} has a solution $\uopt_\alpha$ for all $\alpha \ge 0$.  
If $\alpha > 0$, the solution is unique. If $\alpha = 0$ and 
the operator $T$ is injective, the 
solution of the limit problem \eqref{OCPl} is unique, too.
Note that $T$ is injective in
Example~\ref{exam:poisson} and \ref{exam:timedep}.

If $T$ is not injective, the limit problem might have several solutions.
By $\hat u_0$ we denote \emph{the} solution of the limit 
problem with minimal $U$ norm, i.e. 
$    \hat u_0 = 
   \argmin \twoset{\norm{u}_U}{\text{$u$ solves \eqref{OCPl}}}.$
We close the section by stating a first convergence result, which
in particular shows that the regularized solutions
$\uopt_\alpha$ converge to $\hat u_0$ if $\alpha$ tends to zero. 

More convergence results are obtained in the \textbf{third section}
if a condition on the smoothness
of the limit problem \eqref{OCPl} is fulfilled.
For easy reference, we call this Assumption~\ref{A:sourcestruct}
\emph{source-measure condition} below. 
The main result is Theorem~\ref{T:regmain}, where we show
convergence rates which improve
known ones, see Table~\ref{tab:comp} for a detailed comparison.

The necessity of the just mentioned smoothness conditions to obtain better 
convergence rates is a topic which is discussed
in the \textbf{fourth section}.
We present a new proof of the necessity of the measure condition
which motivates another condition, namely \eqref{E:struct2},
in the special case of bang-bang solutions.

This new condition \eqref{E:struct2} is exploited in the 
\textbf{fifth section}.
We show that the condition implies the
same convergence rates as the source-measure condition.
The new condition is (almost) necessary to obtain these rates.
Finally, it turns out that the new and the old condition coincide if
the limit problem is of certain regularity.

The reason to introduce this new condition \eqref{E:struct2}
is that it leads to an improved
bound on the decay of smoothness in the weak derivative 
of the optimal control
when $\alpha$ tends to zero. This bound is useful to derive improved 
convergence rates for the discretization errors of the regularized 
problem, which we sketch.

The \textbf{last section} is concerned with a numerical example confirming
our theoretical findings.

\section{First results}

\begin{lemm}\label{L:OCPexistence}
The optimal control problem \eqref{OCP} admits 
for fixed $\alpha \ge 0$
at least one solution
$\uopt_\alpha\in U$, which can be characterized
by the first order necessary and sufficient optimality condition
\begin{equation}\label{VarIneq}
   \uopt_\alpha\in\Uad,\quad
   \left( \alpha\uopt_\alpha + \dual B\popt_\alpha,u-\uopt_\alpha\right)_U 
   \ge 0\quad \forall\ u\in\Uad
\end{equation}
where $\dual B$ denotes the adjoint operator of $B$, 
$\yopt_\alpha:=T\uopt_\alpha\in H$ is named \emph{optimal state}, and
the so-called \emph{optimal adjoint state} $\popt_\alpha$ is defined by 
$\popt_\alpha:=\dual S(\yopt_\alpha-z)$. 

If $\alpha >0$ or $T$ is injective, the solution
$\uopt_\alpha$ is unique. The quantities $\yopt_\alpha$ and $\popt_\alpha$
are always unique for given $\alpha\ge 0$ even if $\uopt_0$ is not.
\end{lemm}
\begin{proof}
We have a convex optimization problem with a
weakly lower semicontinuous cost functional on the
non-empty, bounded, closed, and convex set $\Uad$. 
Therefore, classic theory as elaborated, e.g., in
\cite{ekeland-temam}, guarantees existence and uniqueness.
We refer to
\cite[Theorem 1.46, p. 66]{hpuu} or \cite[Satz 2.14]{troeltzsch}
for a proof in our specific setting.

Note that in the case $\alpha = 0$, 
uniqueness of the state $\yopt_0$ follows from the fact that
the cost functional of \eqref{OCPl}
with respect to the state, i.e.  $y\mapsto \frac 12 \norm{y-z}^2_H$,
is strictly convex.
Thus by injectivity of $T$, uniqueness of $\uopt_0$ can be derived
since $\yopt_0 = T\uopt_0$.
\end{proof}

As a consequence of the fact that $\Uad$ is a closed and convex set in
a Hilbert space we have the following lemma.
\begin{lemm}\label{L:orthoproj}
In the case $\alpha > 0$ the variational inequality \eqref{VarIneq}
is equivalent to
\begin{equation}\label{FONC}
\uopt_\alpha = P_{\Uad}\left(-\frac{1}{\alpha}\dual B\popt_\alpha\right)
\end{equation}
where $P_{\Uad}: U \to \Uad$ is the orthogonal projection. 
\end{lemm}
\begin{proof}
See \cite[Corollary 1.2, p. 70]{hpuu} with $\gamma = \frac 1\alpha$.
\end{proof}

We now derive an explicit characterization of optimal controls.
\begin{lemm}\label{L:ucharact}
If $\alpha >0$, then for almost all $x\in \Omega_U$
there holds for the optimal control
\begin{equation}\label{E:uoptpwchar}
  \uopt_\alpha (x)= \begin{cases}   
        a(x)&\text{if $\dual B\popt_\alpha(x)+\alpha a(x) >0$},\\
        -\alpha^{-1}\dual B\popt_\alpha(x)
           &\text{if $\dual B\popt_\alpha(x) 
                        + \alpha \uopt_\alpha(x) = 0$},\\
        b(x)&\text{if $\dual B\popt_\alpha(x)+\alpha b(x) <0$}.
              \end{cases}   
\end{equation}
Suppose $\alpha = 0$ is given. Then any optimal control fulfills a.e.
in $\Omega_U$
\begin{equation}\label{E:bangbang}
  \uopt_0 (x) \begin{cases}   
        =a(x)&\text{if $\dual B\popt_0(x) >0$},\\
        \in [a(x),b(x)]&\text{if $\dual B\popt_0(x) =0$},\\
        =b(x)&\text{if $\dual B\popt_0(x) <0$}.
     \end{cases}
\end{equation}
\end{lemm}
\begin{proof}
Let us first note that the variational inequality \eqref{VarIneq} is 
for $\alpha \ge 0$ equivalent to the following pointwise one:
\begin{equation}\label{VarIneqpw}
   \forall ' x\in\Omega_U\ \forall\ v\in [a(x),b(x)] :
      \left( \alpha\uopt_\alpha(x) + 
    \dual B\popt_\alpha(x),v-\uopt_\alpha(x)\right)_{\mathbb R} \ge 0
\end{equation}
where ``$\forall '$'' denotes ``for almost all''.

This can be shown via a Lebesgue point argument, see the proof of
\cite[Lemma 2.26]{troeltzsch}.
By cases, one immediately derives
\eqref{E:uoptpwchar} and \eqref{E:bangbang} from \eqref{VarIneqpw}. 
\end{proof}
As a consequence of \eqref{E:bangbang} we have:
If $\dual B\popt_0$ vanishes only on a subset of $\Omega_U$
with Lebesgue measure zero, the optimal control
$\uopt_0$ is unique and a \emph{bang-bang solution}:
It takes values only on the
bounds $a$ and $b$ of the admissible set $\Uad$ given in \eqref{E:Uad}.

If the limit problem \eqref{OCPl} admits several solutions,
we by $\hat u_0$ denote the minimal $U$ norm solution, i.e.
\begin{equation}\label{E:minnormsol}
    \hat u_0 = 
   \argmin \twoset{\norm{u}_U}{\text{$u$ solves \eqref{OCPl}}}.
\end{equation}
Note that this minimization problem has a unique solution since
the $U$ norm is strictly convex and the set \{$u$ solves \eqref{OCPl}\}
is non-empty, closed and convex in $U$.

The next Theorem establishes
convergence $\uopt_\alpha \to \hat u_0$ if
$\alpha\to 0$, which is the reason to highlight the
minimal $U$ norm solution among the solutions of \eqref{OCPl}.

\begin{theo}\label{T:invprob}
  For the solution $(\uopt_\alpha,\yopt_\alpha)$
  of \eqref{OCP} with $\alpha > 0$ and any solution $(\uopt_0,\yopt_0)$ 
  of \eqref{OCPl}, there holds
\begin{enumerate}
 \item The optimal control and the optimal state depend continuously on
       $\alpha$. More precisely, the inequality      
       \begin{equation}\label{E:regestinvprob}
              \norm{\yopt_{\alpha'}-\yopt_\alpha}_H^2
      +\alpha' \norm{\uopt_{\alpha'}-\uopt_\alpha}_U^2
        \le (\alpha-\alpha')(\uopt_{\alpha},
                \uopt_{\alpha'}-\uopt_{\alpha})_U
       \end{equation}
       holds for all $\alpha\ge 0$ and all $\alpha'\ge 0$.

   \item The regularized solutions converge to
         the minimal $U$ norm solution $\hat u_0$, i.e.,
         \begin{equation}\label{E:uatou0}
              \norm{\uopt_\alpha - \hat u_0}_U\to 0
                           \quad\text{ if } \alpha\to 0.
         \end{equation}
   \item The optimal state 
         satisfies the rate of convergence
         \begin{equation}\label{E:regconvgen}
         \norm{\yopt_\alpha-\yopt_0}_H 
                 = o(\sqrt{\alpha}).
         \end{equation}
 \end{enumerate}
\end{theo}
\begin{proof}
The Theorem is a collection of classic results from
the theory of linear inverse problems with convex constraints (given
here by $\Uad$) 
taken from \cite[Chapter~5.4]{engl-hanke-neubauer}, see also  
\cite{diss-neubauer}. 
\end{proof}

\section{Refined convergence rates under additional assumptions}

To prove better rates of convergence with respect to $\alpha$, we rely on
the following assumption.
\begin{assu}[{\cite[Assumption 3.1]{wachsmuth2}}]\label{A:sourcestruct}
    Let $\uopt_0$ be a solution of \eqref{OCPl}.
    There exist a set $A\subset \Omega_U$, a function $w\in H$ 
    with $\dual Tw\in L^\infty(\Omega_U)$ and 
    constants $\kappa > 0$ and $C\ge 0$, such that there holds 
    the inclusion  
    \begin{equation*}
    \twoset{x\in\Omega_U}
            { \dual B\popt_0(x)=0 }\allowbreak\subset A^c
    \end{equation*}
    for the complement $A^c=\Omega_U\backslash A$
    of $A$ and in addition    
   \begin{enumerate}
       \item (source condition) 
          \begin{equation}\label{E:source}
               \chi_{A^c} \uopt_0 = \chi_{A^c} P_{\Uad}(\dual Tw).
          \end{equation}
       \item (($\popt_0$-)measure condition)
          \begin{equation}\label{E:struct}
              \forall\ \epsilon > 0:\quad
               \meas(\twoset{x\in A}{0 \le 
                    \abs{\dual B\popt_0(x)} \le \epsilon})
                        \le C\epsilon^\kappa
          \end{equation}
         with the convention that $\kappa :=\infty$ if the left-hand
         side of \eqref{E:struct} is zero for some $\epsilon > 0$.
   \end{enumerate} 
\end{assu}

Source conditions of the form $\uopt_0 = P_{\Uad}(\dual Tw)$ 
are well known 
in the theory of inverse problems with convex constraints, see
\cite{diss-neubauer} and \cite{engl-hanke-neubauer}.
However,
since they are usually posed almost everywhere, thus globally, they
are unlikely to hold in the optimal control setting with $T$ as in, e.g.,
Example~\ref{exam:poisson}, see \cite[p. 860]{wachsmuth1}.

Similar measure conditions were previously used 
for control problems with elliptic PDEs, starting with the
analysis in \cite{wachsmuth1} and \cite{deckelnick-hinze}. 

A condition related to the measure condition was also used
to establish stability results
for bang-bang control problems with autonomous ODEs, see 
\cite[Assumption 2]{felgenhauer2003}. 

In all above-mentioned references, the measure condition
\eqref{E:struct} is assumed to hold with $A=\Omega_U$, thus globally.
Together with formula \eqref{E:bangbang} one immediately
observes that this implies bang-bang controls.

The combination of both conditions in Assumption~\ref{A:sourcestruct} 
was introduced in \cite{wachsmuth2} and also used in \cite{wachsmuth3}. 

In Theorem~\ref{T:regmain} we will show that if a solution $\uopt_0$
of \eqref{OCPl} fulfills Assumption~\ref{A:sourcestruct}, we have
convergence 
$\uopt_\alpha\to\uopt_0$ for $\alpha\to 0$.
From formula \eqref{E:uatou0} in Theorem~\ref{T:invprob}
we conclude $\uopt_0 = \hat u_0$, which means:
If Assumption~\ref{A:sourcestruct} is valid for a solution of \eqref{OCPl},
this solution has to be the minimal $U$ norm solution \eqref{E:minnormsol}. 

Key ingredient in our analysis of the regularization error
is the following lemma,
which has its origin in the proof of \cite[Theorem 3.5]{wachsmuth2}.

\begin{lemm}\label{L:l1reg} 
  Let Assumption~\ref{A:sourcestruct}.2 be valid 
  for a solution $\uopt_0$ of \eqref{OCPl}. 
  Then there holds 
  with some constant $C>0$ independent of $\alpha$ and $u$ 
  \begin{equation}\label{E:l1reg}
       C\norm{u-\uopt_0}_{L^1(A)}^{1+1/\kappa} 
            \le (\dual B\popt_0,u-\uopt_0)_A
            \le (\dual B\popt_0,u-\uopt_0)_U
            \quad\forall\ u\in\Uad
  \end{equation}
  where $(\cdot,\cdot)_A$ and $(\cdot,\cdot)_U$ denote 
  the scalar products in $L^2(A)$ and $U=L^2(\Omega_U)$, respectively.
\end{lemm}
\begin{proof}
  For $\epsilon > 0$ we define 
  $B_\epsilon := \twoset{x\in A}{\abs{\dual B\popt_0}\ge\epsilon}$.
  Using the (pointwise) optimality condition \eqref{VarIneqpw} and
  Assumption~\ref{A:sourcestruct}.2, we conclude for some $u\in\Uad$
  \begin{equation*}
  \begin{aligned}
    \int_{\Omega_U} (\dual B\popt_0,u-\uopt_0)_{\mathbb R}
    &= \int_{\Omega_U} \abs{\dual B\popt_0}\abs{u-\uopt_0} 
     \ge \int_A \abs{\dual B\popt_0}\abs{u-\uopt_0} \\
    &\ge \epsilon \norm{u-\uopt_0}_{L^1(B_\epsilon)}\\
    &\ge \epsilon \norm{u-\uopt_0}_{L^1(A)}
        -\epsilon \norm{u-\uopt_0}_{L^1(A\backslash B_\epsilon)}\\ 
    &\ge \epsilon \norm{u-\uopt_0}_{L^1(A)}
        -\epsilon \norm{u-\uopt_0}_{L^\infty(\Omega_U)}
             \meas(A\backslash B_\epsilon) \\ 
    &\ge \epsilon \norm{u-\uopt_0}_{L^1(A)}
        -c\epsilon^{\kappa + 1} \norm{u-\uopt_0}_{L^\infty(\Omega_U)}
 \end{aligned}
 \end{equation*}
 where without loss of generality $c>1$.
 
 Setting $\epsilon:=c^{-2/\kappa}
       \norm{u-\uopt_0}_{L^1(A)}^{1/\kappa}
              \norm{u-\uopt_0}_{L^\infty(\Omega_U)}^{-1/\kappa}$
 yields
 \begin{equation*}
 \begin{aligned}
    \int_A (\dual B\popt_0,u-\uopt_0)_{\mathbb R}
      &\ge c^{-2/\kappa} (1-\frac 1c) 
           \norm{u-\uopt_0}_{L^\infty(\Omega_U)}^{-1/\kappa}
           \norm{u-\uopt_0}_{L^1(A)}^{1+1/\kappa}\\ 
      &\ge c^{-2/\kappa} (1-\frac 1c) 
           \norm{b-a}_{L^\infty(\Omega_U)}^{-1/\kappa}
           \norm{u-\uopt_0}_{L^1(A)}^{1+1/\kappa}\\ 
 \end{aligned}
 \end{equation*}
  by the definition of $\Uad$.
\end{proof}
With the previous Lemma, we can now improve the inequality
\eqref{E:regestinvprob} (setting there $\alpha:=0$)
from general inverse problem theory,
since the error in the control in the $L^1$ norm now appears on the
left-hand side with a factor C>0 independent of $\alpha$. This is
in contrast to the error in the $L^2$ norm.

\begin{lemm}\label{L:l1regmore}
   Let Assumption~\ref{A:sourcestruct}.2 hold (with possibly 
   $\meas (A)=0$)
  for a solution $\uopt_0$ of \eqref{OCPl}. 
   Then there holds for some $C>0$ independent of $\alpha$
       \begin{multline*}
              \norm{\yopt_{\alpha}-\yopt_0}_H^2
      + C      \norm{\uopt_{\alpha}-\uopt_0}_{
                          L^1(A)}^{1+1/\kappa}
      + \alpha \norm{\uopt_{\alpha}-\uopt_0}_U^2\\
        \le \alpha (\uopt_0, \uopt_0-\uopt_{\alpha})_U
           \quad\forall\ \alpha >0.
       \end{multline*}
\end{lemm}
\begin{proof}
  Adding the necessary condition for $\uopt_\alpha$ 
  \eqref{VarIneq} with $u:=\uopt_0$, i.e.,
  \begin{equation*}
     0 \le \left( \alpha\uopt_\alpha + \dual B\popt_\alpha,
         \uopt_0-\uopt_\alpha\right)_U ,
  \end{equation*}
  to the estimate \eqref{E:l1reg} of Lemma~\ref{L:l1reg}
  with $u:=\uopt_\alpha$, we get
  \begin{equation*}
  \begin{aligned}
    C\norm{\uopt_{\alpha}-\uopt_0}_{L^1(A)}^{1+1/\kappa}
       &\le ( \dual B(\popt_\alpha-\popt_0),
         \uopt_0-\uopt_\alpha )_U 
            + \alpha (\uopt_\alpha,\uopt_0-\uopt_\alpha)_U \\
       &\le -\norm{\yopt_\alpha-\yopt_0}_H^2 
            + \alpha (\uopt_{\alpha}-\uopt_0,\uopt_0-\uopt_{\alpha})_U\\
       &\quad\quad\quad
            + \alpha (\uopt_0,\uopt_0-\uopt_\alpha)_U\\
       &\le -\norm{\yopt_\alpha-\yopt_0}_H^2 
            - \alpha \norm{\uopt_{\alpha}-\uopt_0}_U^2
            + \alpha (\uopt_0,\uopt_0-\uopt_\alpha)_U.
  \end{aligned}
  \end{equation*}
\end{proof}
The following Lemma is extracted from the proof of 
\cite[Lemma 3.2]{wachsmuth2}. It shows how the source
condition (Assumption~\ref{A:sourcestruct}.1) is taken into
account to reduce the error estimate to the set $A$.
\begin{lemm}\label{L:sourcecest}
   Let Assumption~\ref{A:sourcestruct}.1 (source condition) be
   satisfied
  for a solution $\uopt_0$ of \eqref{OCPl}. 
  Then there holds with a constant $C>0$
   \begin{equation*}
          (\uopt_0, \uopt_0-u)_U
               \le C( \norm{T(u-\uopt_0)}_H 
              +\norm{u-\uopt_0}_{L^1(A)}  )
           \quad\forall\ u\in\Uad.
   \end{equation*}
\end{lemm}
\begin{proof}
    The source condition is equivalent to
    \[ 
       0\le (\chi_{A^c}(\uopt_0 - \dual Tw),u-\uopt_0)_U
         \quad\forall\ u\in\Uad.   
    \]
    Using this representation, we can estimate
    \begin{align*}
          (\uopt_0, \uopt_0-u)_U
      &\le \left(\chi_{A^c}\dual Tw,\uopt_0-u\right)_U
           +\left(\chi_A \uopt_0,\uopt_0-u\right)_U\\
      &\le \left(w,T(\uopt_0-u)\right)_H
             +\left(-\dual Tw+\uopt_0,
                \chi_A\left(\uopt_0-u\right)\right)_U.
    \end{align*}
    Since $\dual Tw\in L^\infty(\Omega_U)$, we get the claim.
\end{proof}

Using this Lemma, we can now state regularization error estimates. 
We consider different situations with respect to the
fulfillment of parts of Assumption~\ref{A:sourcestruct}.

\begin{theo}\label{T:regmain}
  For the regularization error there holds with positive constants 
  $c$ and $C$ indepent of $\alpha > 0$ the following, where
  $\uopt_0$ is in fact $\hat u_0$ as noted before Lemma~\ref{L:l1reg}.
  \begin{enumerate}
  \item The error in the optimal state fulfills the rate of 
        convergence
        \begin{equation*}
        \norm{\yopt_\alpha-\yopt_0}_H 
              = o(\sqrt{\alpha}).
        \end{equation*}
  \item Let Assumption~\ref{A:sourcestruct}.1 be satisfied with 
        $\meas(A)=0$ (source condition holds a.e. on the domain)
  for a solution $\uopt_0$ of \eqref{OCPl}. 
        Then the optimal control converges with the rate
       \begin{equation}
             \norm{\uopt_{\alpha}-\uopt_0}_U \le C\sqrt{\alpha},
       \end{equation}
        and the optimal state converges with the improved rate 
       \begin{equation}\label{E:regconvbetter}
              \norm{\yopt_{\alpha}-\yopt_0}_H \le C\alpha.
       \end{equation}
  \item Let Assumption~\ref{A:sourcestruct}.2 be satisfied with 
        $\meas(A^c)=0$ (measure condition holds a.e. on the domain) 
        for a solution $\uopt_0$ of \eqref{OCPl}.
        From \eqref{E:bangbang} we conclude that 
        $\uopt_0$ is the unique solution of \eqref{OCPl}. 
        Then the estimates
     \begin{align}
       \norm{\uopt_{\alpha}-\uopt_0}_{L^1(\Omega_U)}
                              &\le C\alpha^{\kappa}
                 \label{E:regconvumeas} \\
       \norm{\uopt_{\alpha}-\uopt_0}_U &\le C\alpha^{\kappa/2}
                 \label{E:regconvumeas2} \\
       \norm{\yopt_{\alpha}-\yopt_0}_H &\le C\alpha^{(\kappa+1)/2}
                 \label{E:regconvymeas}
     \end{align} 
        hold true.

        If furthermore $\kappa > 1$ holds and in addition
        \begin{equation}\label{E:dualTlinfty}
            \dual T:\range(T)\to L^\infty(\Omega_U)
            \quad\quad\text{exists and is continuous,}
        \end{equation}
        we can improve 
        \eqref{E:regconvymeas} to
       \begin{equation}\label{E:regconvymeasimp}
           \norm{\yopt_{\alpha}-\yopt_0}_H \le C\alpha^{\kappa}.
       \end{equation}
  \item Let Assumption~\ref{A:sourcestruct} be satisfied with 
        $\meas(A)\cdot\meas(A^c) > 0$
        (source and measure condition on parts of the domain) 
        for a solution $\uopt_0$ of \eqref{OCPl} 
        and let in addition $\alpha < 1$.
        Then the estimates
        \begin{align}
          \norm{\uopt_{\alpha}-\uopt_0}_{L^1(A)}
                    &\le C\alpha^{\min(\kappa,\,\frac 2{1+1/\kappa})}
		 \label{E:regconvusourcemeas} \\
          \norm{\uopt_{\alpha}-\uopt_0}_U  
                    &\le C\alpha^{\min(\kappa,\,1)/2}
 		\label{E:regconvusourcemeas2} \\
          \norm{\yopt_{\alpha}-\yopt_0}_H 
                    &\le C\alpha^{\min((\kappa+1)/2,\,1)}
 		\label{E:regconvysourcemeas}
        \end{align}
        hold true.

        If furthermore $\kappa > 1$ and \eqref{E:dualTlinfty} hold,
        we have the improved estimate
        \begin{equation}\label{E:regconvusourcemeasimp}
          \norm{\uopt_{\alpha}-\uopt_0}_{L^1(A)}
                \le C\alpha^\kappa.
       \end{equation}
  \end{enumerate}
\end{theo}
\begin{proof}
In this proof, we denote by $C_1, \ldots, C_4$ positive constants.

1. The estimate is just a repetition of \eqref{E:regconvgen}.

3. Let us recall the estimates of Lemma~\ref{L:l1regmore}, i.e.,
   \begin{equation}\label{E:l1regmorecopy}
              \norm{\yopt_{\alpha}-\yopt_0}_H^2
      + C      \norm{\uopt_{\alpha}-\uopt_0}_{
                          L^1(A)}^{1+1/\kappa}
      + \alpha \norm{\uopt_{\alpha}-\uopt_0}_U^2
        \le \alpha (\uopt_0, \uopt_0-\uopt_{\alpha})_U.
   \end{equation}
  By Young's inequality we can estimate with a constant $\hat C>0$ 
      \begin{equation}\label{E:youngkappa}
        \hat C\alpha \norm{\uopt_{\alpha}-\uopt_0}_{L^1(A)}
           \le \tilde{C}\alpha^{\kappa+1} 
                + \frac{C}{2}
              \norm{\uopt_{\alpha}-\uopt_0}_{
                    L^1(A)}^{1+1/\kappa}
      \end{equation}
  where
  $C$ is the same constant as in \eqref{E:l1regmorecopy}
  and $\tilde C=\tilde C(C,\hat C,\kappa)$ is the constant 
  from Young's inequality.

  If $A=\Omega_U$ up to a set of measure zero,
  we can combine both estimates 
  taking $\hat C:=\norm{\uopt_0}_{L^\infty}$,
  and move the second summand of \eqref{E:youngkappa} to the left.
  This yields the claim since
  \begin{equation*}  
       \frac{\kappa + 1}{1+1/\kappa} = \kappa .
  \end{equation*}
  The improved estimate \eqref{E:regconvymeasimp} can be obtained
  with the help of \eqref{E:regconvumeas} as follows
  \begin{equation*}
  \begin{aligned}
       \norm{\yopt_{\alpha}-\yopt_0}_H^2 
    &= (\dual T(\yopt_{\alpha}-\yopt_0),\uopt_\alpha-\uopt_0)_U 
    \le C_1 \norm{\dual T(\yopt_{\alpha}-\yopt_0)}_{L^\infty} 
               \norm{\uopt_\alpha-\uopt_0}_{L^1}\\
    &\le C_2  \norm{\yopt_{\alpha}-\yopt_0}_H
               \norm{\uopt_\alpha-\uopt_0}_{L^1}
    \le C_3  \norm{\yopt_{\alpha}-\yopt_0}_H\,
               \alpha^\kappa. 
   \end{aligned}
   \end{equation*}

2.+4. We combine \eqref{E:l1regmorecopy}
  with the estimate of Lemma~\ref{L:sourcecest} (with
  $u:=\uopt_\alpha$), invoke Cauchy's inequality and get
       \begin{multline*}
              \norm{\yopt_{\alpha}-\yopt_0}_H^2
      + C      \norm{\uopt_{\alpha}-\uopt_0}_{
                          L^1(A)}^{1+1/\kappa}
      + \alpha \norm{\uopt_{\alpha}-\uopt_0}_U^2\\
        \le \alpha (\uopt_0, \uopt_0-\uopt_{\alpha})_U
        \le C_1\alpha ( \norm{\yopt_\alpha-\yopt_0}_H 
              +\norm{\uopt_{\alpha}-\uopt_0}_{L^1(A)})\\
      \le C_2\alpha^2 + \frac 12\norm{\yopt_\alpha-\yopt_0}_H^2
             + C_1\alpha 
              \norm{\uopt_{\alpha}-\uopt_0}_{L^1(A)}.
       \end{multline*}
  We now move the second addend to the left. 

  If $\meas(A)=0$ (case 2.), we are done. 
  Otherwise (case 4.) we continue estimating, making use of
  \eqref{E:youngkappa}, to get 
   \begin{equation*}
              \norm{\yopt_{\alpha}-\yopt_0}_H^2
      + C      \norm{\uopt_{\alpha}-\uopt_0}_{
                          L^1(A)}^{1+1/\kappa}
      + \alpha \norm{\uopt_{\alpha}-\uopt_0}_U^2
       \le C_3\alpha^{\min(2,\kappa+1)},
   \end{equation*}
  from which the claim follows. 

  To establish formula \eqref{E:regconvusourcemeasimp},
  we integrate \eqref{VarIneqpw} over $A$, taking 
  $v:=\uopt_0(x)$, to end up with
  \begin{equation*}
     0 \le \left( \alpha\uopt_\alpha + \dual B\popt_\alpha,
         \uopt_0-\uopt_\alpha\right)_A .
  \end{equation*}
  By $(\cdot,\cdot)_A$ we again denote the scalar product in $L^2(A)$.

  We add this inequality
  to the estimate \eqref{E:l1reg} of Lemma~\ref{L:l1reg}
  with $u:=\uopt_\alpha$, to get
  \begin{equation*}
    C\norm{\uopt_{\alpha}-\uopt_0}_{L^1(A)}^{1+1/\kappa}
       \le ( \dual B(\popt_\alpha-\popt_0),
         \uopt_0-\uopt_\alpha )_A 
            + \alpha (\uopt_\alpha,\uopt_0-\uopt_\alpha)_A.
  \end{equation*}
  Making use of \eqref{E:dualTlinfty}
  and the convergence rate \eqref{E:regconvysourcemeas} 
  with $\kappa > 1$, we conclude
  \begin{equation*}
     \norm{\dual B(\popt_0-\popt_\alpha)}_{L^\infty(\Omega_U)} 
    = \norm{\dual T(\yopt_0-\yopt_\alpha)}_{L^\infty(\Omega_U)}  
   \le C_1 \norm{\yopt_0-\yopt_\alpha}_H \le C_2\alpha.
  \end{equation*}
  Since $\uopt_\alpha\in L^\infty(\Omega_U)$ by \eqref{E:Uad},
  combining both estimates gives 
  \begin{equation*}
  \begin{aligned}
    C\norm{\uopt_{\alpha}-\uopt_0}_{L^1(A)}^{1+1/\kappa}
       &\le C_3
          (\norm{\dual B(\popt_\alpha-\popt_0)}_{L^\infty(A)} + \alpha)
                   \norm{ \uopt_0-\uopt_\alpha }_{L^1(A)} \\
       &\le C_4\alpha \norm{ \uopt_0-\uopt_\alpha }_{L^1(A)} .
  \end{aligned}
  \end{equation*}
  Dividing the expression by the norm on the right and
  taking the $\kappa$th power, we are done.
\end{proof}
Some remarks on the previous theorem are in order.

Let us compare the first with the other cases, where
Assumption~\ref{A:sourcestruct} is taken (partially) into account.
In all cases, we get an improved convergence rate for the
optimal state.

The second case replicates well known estimates from the
theory of inverse problems with convex constraints, see, e.g.,
\cite{diss-neubauer} and \cite[Theorem 5.19]{engl-hanke-neubauer}.
However, as indicated in the discussion
after Assumption~\ref{A:sourcestruct}, this situation is
unlikely to hold in the context of optimal control problems.

Concerning the ``min''-functions in the estimates of case 4, we note
that the left argument is chosen if $\kappa < 1$, the right 
one if $\kappa > 1$. In the case $\kappa = 1$, both
expressions coincide. Thus the worse part of 
Assumption~\ref{A:sourcestruct} with respect to the rates of cases 2 and
3 dominates the convergence behavior of the regularization errors
in the mixed situation of case 4 on the whole domain $\Omega_U$. 
This, however, is not the case locally on $A$, as 
\eqref{E:regconvusourcemeasimp} shows.

As mentioned after Assumption~\ref{A:sourcestruct}, case 3
implies bang-bang controls. 

The condition \eqref{E:dualTlinfty} is fulfilled for 
Example~\ref{exam:poisson} 
since $\dual T:\H\to H^2(\Omega)\cap \HI \hookrightarrow \Loo$
by well-known regularity theory
and Sobolev imbeddings, see, e.g., \cite{evans},
if $\Omega$ is sufficiently regular.

For Example~\ref{exam:timedep},
condition \eqref{E:dualTlinfty} is also valid, see \cite[p. 24]{dissnvd}.

\begin{table}[hbt]
\begin{center}
\begin{tabular}{c|c|c|c|c}
\hline
quantity $\le C\alpha^r$
& here
& br
& there
& assumptions, source\\
\hline
$\norm{\uopt_{\alpha}-\uopt_0}_{L^1(A)}$
&$r=\kappa$ 
&$\leftarrow$
&$r=\frac{\kappa}{2-\kappa}$
&$\kappa <1$    \\
&&&& by \eqref{E:regconvumeas}/\eqref{E:regconvusourcemeas}\\

\rowcolor[gray]{.9}
$\norm{\uopt_{\alpha}-\uopt_0}_{L^1(A)}$
&$r=\kappa$ 
&$=$
&$r=\kappa$ 
&$\kappa = 1$ or (3. and $\kappa > 1$)    \\
\rowcolor[gray]{.9}
&&&& by \eqref{E:regconvumeas} or \eqref{E:regconvusourcemeas}\\ 

$\norm{\uopt_{\alpha}-\uopt_0}_{L^1(A)}$
&$r=\kappa$ 
&$\leftarrow$
&$r=\frac{\kappa+1}{2}$
&4. and $\kappa > 1$\\
&&&& by \eqref{E:regconvusourcemeasimp} \\
\hline
$\norm{\uopt_{\alpha}-\uopt_0}_{U}$
&$r=\frac{\kappa}{2}$ 
&$\leftarrow$
&$r=\frac{\kappa}{2(2-\kappa)}$
&$\kappa < 1$\\
&&&&  by \eqref{E:regconvumeas2} or \eqref{E:regconvusourcemeas2}\\

\rowcolor[gray]{.9}
$\norm{\uopt_{\alpha}-\uopt_0}_{U}$
&$r=\frac{\kappa}{2}$ 
&$=$
&$r=\frac{\kappa}{2}$ 
&$\kappa=1$ or (3. and  $\kappa > 1$)\\
\rowcolor[gray]{.9}
&&&&  by \eqref{E:regconvumeas2}\\

$\norm{\uopt_{\alpha}-\uopt_0}_{U}$
&$r=\frac{1}{2}$ 
&$=$
&$r=\frac{1}{2}$ 
&4. and $\kappa > 1$\\
&&&&  by \eqref{E:regconvusourcemeas2} \\
\hline
$\norm{\yopt_{\alpha}-\yopt_0}_H $
&$r=\frac{\kappa+1}{2}$
&$\leftarrow$
&$r=\frac{1}{2-\kappa}$
&$\kappa < 1$\\
&&&& by \eqref{E:regconvymeas} or \eqref{E:regconvysourcemeas}\\

\rowcolor[gray]{.9}
$\norm{\yopt_{\alpha}-\yopt_0}_H $
&$r=1$
&$=$
&$r=1$
&$\kappa = 1$  or (4. and $\kappa > 1$)\\
\rowcolor[gray]{.9}
&&&& by \eqref{E:regconvymeas} or \eqref{E:regconvysourcemeas}\\

$\norm{\yopt_{\alpha}-\yopt_0}_H $
&$r=\kappa$
&$\leftarrow$
&$r=\frac{\kappa+1}{2}$
&3. and $\kappa > 1$\\
&&&&  by \eqref{E:regconvymeasimp}\\
\end{tabular}
\end{center}
\caption{
         Comparison of convergence rates given in
         Theorem~\ref{T:regmain}.3+4 (``here'')
         with \cite[Theorem 3.2]{wachsmuth2} (``there''), 
         assuming always \eqref{E:dualTlinfty}. 
         The column ``br'' points to the \emph{better rate} 
         (i.e. larger $r$) unless both coincide ($=$).
         We abbreviate by ``3.'' and ``4.'' the assumptions of
         Theorem~\ref{T:regmain}.3 and 4, respectively. 
        }
\label{tab:comp}
\end{table}
Let us finally compare in Table~\ref{tab:comp}
the cases 3 and 4 with the
convergence results of \cite[Theorem 3.2]{wachsmuth2} to point out
which rates stated above are improved.
Note for comparison, that \eqref{E:dualTlinfty} is always assumed
in \cite[Theorem 3.2]{wachsmuth2} and $p_\alpha$ there is
$\dual B\popt_\alpha$ here. If we assume \eqref{E:dualTlinfty},
we can estimate 
$
\norm{\dual B(\popt_0-\popt_\alpha)}_{L^\infty} 
=\norm{\dual T(\yopt_0-\yopt_\alpha)}_{L^\infty}  
\le C \norm{\yopt_0-\yopt_\alpha}_H
$, and combine this with 
\eqref{E:regconvymeas}, \eqref{E:regconvymeasimp}, or 
\eqref{E:regconvysourcemeas}. 
Since 
in \cite[Theorem 3.2]{wachsmuth2}
the convergence rates for 
$\norm{p_0-p_\alpha}_{L^\infty}$
are obtained in the same way,
comparing the state rates 
gives the same results as comparing 
$\norm{\dual B(\popt_0-\popt_\alpha)}_{L^\infty}$ with
$\norm{p_0-p_\alpha}_{L^\infty}$.
We therefore omit the latter.

\clearpage 

\section{Necessity of the additional assumptions}

Let us now consider the question of necessity of 
Assumption~\ref{A:sourcestruct} to obtain convergence rates,
thus a converse of Theorem~\ref{T:regmain}.

We first show that a convergence rate
$\norm{\yopt_\alpha - \yopt_0}_H\le C \alpha$
implies the source condition \eqref{E:source} 
to hold at least on 
          $\twoset{x\in\Omega_U}{ \dual B\popt_0(x)=0 }$.

The following Theorem is a for our purposes simplified version
of \cite[Theorem 4]{wachsmuth3}.
It resembles a necessity result known from inverse problem
theory, see, e.g., \cite[Theorem 5.19]{engl-hanke-neubauer} or 
\cite{diss-neubauer}. However, in inverse problems, the condition
$T\uopt_0 = z$ is typically assumed.

\begin{theo}\label{T:necsource}
  Let $\hat u_0$ be the minimal $U$ norm solution of \eqref{OCPl}
  defined in \eqref{E:minnormsol}. 
  If we assume a convergence rate
  $\norm{\yopt_\alpha - \yopt_0}_H=\mathcal{O}(\alpha)$, then there
  exists a function $w\in H$ such that 
        $\hat u_0 = P_{\Uad} (\dual Tw)$
  holds pointwise a.e. on 
  \begin{equation}\label{E:Kset}
        K:=\twoset{x\in\Omega_U}{ \dual B\popt_0(x)=0}.
  \end{equation}
  Thus \eqref{E:source} holds on $K$ instead of $A^c$.
  
  If even $\norm{\yopt_\alpha - \yopt_0}_H=o(\alpha)$,
  then $\hat u_0$ vanishes on $K$.
\end{theo}
\begin{proof}
We integrate the necessary condition \eqref{VarIneqpw} 
over $K$ to obtain
\[
  0 \le \left( \alpha\uopt_\alpha +
        \dual TT\left(\uopt_\alpha-\hat u_0\right),
            u-\uopt_\alpha \right)_K\quad\quad\forall\ u\in\Uad. 
\]
Dividing the expression by $\alpha$ and taking the limit
we get with the help of \eqref{E:uatou0} the inequality
\[
    0 \le (\dual T\dot{y_0} + \hat u_0 , u - \hat u_0)_K
          \quad\quad\forall\ u\in\Uad 
\]
for any weak subsequential limit $\dot{y_0}$ of 
$\frac 1\alpha (\yopt_\alpha - \yopt_0)$, which exists due to the
assumption of the Theorem. 

Taking $w:=-\dot{y_0}$, we obtain 
the equation $\hat u_0(x) = P_{[a(x),b(x)]}(w(x))$ pointwise on $K$
by varying $u$.
Since $P_{\Uad}$ acts pointwise, we get the claim.

The second assertion follows from the equality
$\dot{y_0} = 0$ in case of 
$\norm{\yopt_\alpha - \yopt_0}_H=o(\alpha)$.
\end{proof}

We next show that if
\eqref{E:dualTlinfty} and $\kappa > 1$ hold true,
convergence as in 
Theorem~\ref{T:regmain}.3
implies the measure condition \eqref{E:struct}.

\begin{theo}\label{T:necmeasure}
  Let us assume
  \begin{equation}\label{E:zerosetinAc}
      \exists\ A\subset \Omega_U:\quad
      \twoset{x\in\Omega_U}{\dual B\popt_0(x)=0}\subset A^c.
  \end{equation}
  Let us further assume the \emph{$\sigma$-condition}
  \begin{equation}\label{E:sigmacond}
    \exists\ \sigma > 0\ \forall'\ x\in\Omega_U:\quad
          a \le -\sigma < 0 < \sigma \le b
  \end{equation}
  where ``$\forall '$'' denotes ``for almost all''.

  If $\kappa > 1$ and convergence rates
  $\norm{\uopt_\alpha - \uopt_0}_{L^p(A)}^p 
    + \norm{\dual B(\popt_\alpha - \popt_0)}
          _{L^\infty(A)} 
      \le C\alpha^\kappa$
  are known to hold for 
  a solution $\uopt_0$ of \eqref{OCPl} and
  some real $p\ge 1$,
  then the measure condition \eqref{E:struct} 
  from Assumption~\ref{A:sourcestruct} is fulfilled.
\end{theo}
\begin{proof}
  Let us introduce the sets
  \begin{align*}
     A_0 &:= \twoset{x\in A}{-\dual B\popt_0 < 0 \text{ and } 
     \alpha a \ge -\dual B\popt_\alpha},\\
     A_1 &:= \twoset{x\in A}{-\dual B\popt_0 < 0 \text{ and } 
     \alpha a < -\dual B\popt_\alpha < \alpha b},\\
     A_2 &:= \twoset{x\in A}{  -\dual B\popt_0 
                < 0 < \alpha b \le
           -\dual B\popt_\alpha},\\
     A_3 &:= \twoset{x\in A}{ -\dual B\popt_0 > 0 
                 \text{ and } 
        \alpha a < -\dual B\popt_\alpha < \alpha b},\\
     A_4 &:= \twoset{x\in A}{ -\dual B\popt_0 > 0 
             > \alpha a \ge -\dual B\popt_\alpha },
         \quad\text{and}\\
     A_5 &:= \twoset{x\in A}{ -\dual B\popt_0 > 0 
                 \text{ and }
             \alpha b \le -\dual B\popt_\alpha } .
  \end{align*}
  We also need two subsets of $A_1$ and $A_3$, respectively, namely
  by \eqref{E:sigmacond}
  \begin{align*}
     \tilde A_1 &:= \twoset{x\in A}{-\dual B\popt_0 < 0 \text{ and } 
     -\alpha \frac\sigma 2 \le -\dual B\popt_\alpha 
                       \le \alpha \frac\sigma 2}
            \subset A_1, \quad\text{and}\\
     \tilde A_3 &:= \twoset{x\in A}{ -\dual B\popt_0 > 0 \text{ and } 
     -\alpha \frac\sigma 2 \le -\dual B\popt_\alpha 
                       \le \alpha \frac\sigma 2}
               \subset A_3.
  \end{align*}
  From \eqref{E:zerosetinAc} we conclude
  $A = A_0 \cup A_1 \cup A_2 \cup A_3 \cup A_4\cup A_5$, and
  from Lemma~\ref{L:ucharact} we infer
  \begin{equation}\label{E:palphameas}
  \begin{aligned}
      \int_A \abs{\uopt_0-\uopt_\alpha}^p 
      &=
      \int_{A_1} \abs{a+\alpha^{-1}\dual B \popt_\alpha}^p
     +\int_{A_3} \abs{b+\alpha^{-1}\dual B \popt_\alpha}^p
     +\int_{A_2\cup A_4} \abs{a-b}^p\\
      &\ge
      \int_{A_1} \abs{a+\alpha^{-1}\dual B \popt_\alpha}^p
     +\int_{A_3} \abs{b+\alpha^{-1}\dual B \popt_\alpha}^p\\
      &\ge
      \int_{\tilde A_1} \abs{a+\alpha^{-1}\dual B \popt_\alpha}^p
     +\int_{\tilde A_3} \abs{b+\alpha^{-1}\dual B \popt_\alpha}^p\\
      &\ge
      (\frac \sigma 2)^p \meas(
          \twoset{x\in A}{\abs{\dual B\popt_\alpha} \le
             \frac \sigma 2\alpha } ).
  \end{aligned}
  \end{equation}
  Note for the last step that 
  $\tilde A_1 \cup \tilde A_3
      = \twoset{x\in A}{\abs{\dual B\popt_\alpha} \le
             \frac \sigma 2\alpha } 
  $
  due to \eqref{E:zerosetinAc}.

  From 
  $\norm{\uopt_\alpha - \uopt_0}_{L^p(A)}^p 
      \le C\alpha^\kappa$
  and \eqref{E:palphameas} we conclude
  \[
     \meas(\twoset{x\in A}{\abs{\dual B\popt_\alpha}
          \le C_1\alpha }) \le C_2\alpha^\kappa.
  \]
  Since $\kappa > 1$ and 
     $\norm{\dual B(\popt_\alpha - \popt_0)}_{L^\infty(A)} 
      \le C\alpha^\kappa$,
  we get for some arbitrarily chosen $x\in A$ with 
  $\abs{\dual B\popt_0(x)} \le \alpha C_1/2$ the
  estimate
  \[  
     \abs{\dual B\popt_\alpha(x)} 
     \le \abs{\dual B\popt_0(x)} 
         + \abs{\dual B(\popt_\alpha-\popt_0)(x)}  
     \le \frac{C_1}{2} (\alpha + \alpha^{\kappa -\epsilon}) 
     \le C_1\alpha
  \]
 for some sufficiently small 
 $\epsilon = \epsilon(C_1,\kappa) > 0$.
 Consequently, we have
 \[
     \meas(\twoset{x\in A}{
          \abs{\dual B\popt_0} \le \frac{C_1}{2} \alpha  
         }) \le C_2\alpha^\kappa.
 \]
\end{proof}

Concerning the previous Theorem, let us mention the related result  
\cite[Theorem 8]{wachsmuth3}. It has the same implication,
but assumes \eqref{E:regconvumeas2} and \eqref{E:regconvymeas},
which imply the prerequisites of Theorem~\ref{T:necmeasure} in case
of \eqref{E:dualTlinfty}. 

For the case $\kappa \le 1$, it is an open question whether
the previous Theorem (and likewise \cite[Theorem 8]{wachsmuth3})
is valid.

Note that the $\sigma$-condition \eqref{E:sigmacond} is a
strengthening of the condition ``$a \le 0 \le b$ almost everywhere''.
For \eqref{OCPl}, the problem we finally want to solve, this weaker
assumption can always be met by a simple transformation of the variables.

\section{Bang-bang solutions}

In this section, 
we introduce at first a second measure condition and show that
it implies the same convergence results as
in Theorem~\ref{T:regmain}.3, thus might replace
the $\popt_0$-measure condition 
\eqref{E:struct} from Assumption~\ref{A:sourcestruct}.

We analyze necessity of the condition to obtain convergence rates
and show that for bang-bang solutions fulfilling
\begin{equation}\label{E:bangbangsol}
    \meas(\twoset{x\in\Omega_U}{ \dual B\popt_0(x)=0 })=0,
\end{equation}
both measure conditions coincide.

Note that \eqref{E:bangbangsol} by \eqref{E:bangbang} implies
uniqueness of the solution $\uopt_0$ of \eqref{OCPl}.

\begin{defi}[$\popt_\alpha$-measure condition]
   If for the set
   \begin{equation}\label{E:setialpha}
     I_\alpha:=\twoset{x\in \Omega_U}
                 {\alpha a < -\dual B\popt_\alpha < \alpha b}
   \end{equation}
   the condition
   \begin{equation}\label{E:struct2}
     \exists\ \bar\alpha >0\ \forall\ 0<\alpha<\bar\alpha:
       \quad   \meas(I_\alpha)\le C\alpha^\kappa
   \end{equation}
   holds true 
   (with the convention that $\kappa := \infty$ if the measure in 
   \eqref{E:struct2} is zero for all $0<\alpha<\bar\alpha$),
   we say that the 
   \emph{$\popt_\alpha$-measure condition} is fulfilled. 
\end{defi}

The equality in the estimate \eqref{E:palphameas} 
from the proof of Theorem~\ref{T:necmeasure} 
shows that if the $\popt_\alpha$-measure condition holds 
and we assume the additional condition 
$ \meas(A_2\cup A_4)\le C\alpha^\kappa $
(with $A_i$ as in that proof), 
we get the convergence rate
  $ \norm{\uopt_\alpha - \uopt_0}_{L^p(\Omega_U)}^p 
       \le C\alpha^\kappa 
  $
for each $1\le p < \infty$ given \eqref{E:bangbangsol}. 

Interestingly, these additional conditions are not needed
to obtain convergence in the control, as we will now show.
\begin{theo}\label{T:convpalpha}
  If the $\popt_\alpha$-measure condition \eqref{E:struct2} and
  the $\sigma$-condition \eqref{E:sigmacond}
  are fulfilled, the convergence rates
  \begin{equation}\label{E:convpalpharates}
        \norm{\uopt_{\alpha}-\uopt_0}_{L^1(\Omega_U)}
                           \le C\alpha^{\kappa}
      \quad\text{and}\quad
       \norm{\yopt_{\alpha}-\yopt_0}_H \le C\alpha^{(\kappa+1)/2}
  \end{equation}
  hold true for any solution $\uopt_0$ of \eqref{OCPl}. 
  
  If in addition $\kappa > 1$ and \eqref{E:dualTlinfty} is fulfilled,
  we have the improved estimate
  \begin{equation}\label{E:convpalpharateimp}
       \norm{\yopt_{\alpha}-\yopt_0}_H \le C\alpha^{\kappa}.
  \end{equation}
\end{theo}
\begin{proof}
   Let $u\in\Uad$ be arbitrarily chosen.
   For the active set $I_\alpha^c$ of $\popt_\alpha$,
   which is the complement of the inactive set $I_\alpha$ defined in
   \eqref{E:setialpha}, we have by Lemma~\ref{L:ucharact},
   making use of the $\sigma$-condition \eqref{E:sigmacond}, the estimate
   \begin{equation}\label{E:skpest1}
      (\dual B\popt_\alpha,u-\uopt_\alpha)_{I_\alpha^c}
        = \int_{I_\alpha^c}
             \abs{\dual B\popt_\alpha}\abs{u-\uopt_\alpha}
        \ge \sigma \alpha \norm{u-\uopt_\alpha}_{L^1(I_\alpha^c)}.
   \end{equation} 
   Invoking the $\popt_\alpha$-measure condition \eqref{E:struct2}, 
   we get on the inactive set itself the estimate
   \begin{equation}\label{E:skpest2}
      \abs{(\dual B\popt_\alpha,u-\uopt_\alpha)_{I_\alpha}}
        \le C\alpha \norm{u-\uopt_\alpha}_{L^1(I_\alpha)}
        \le CC_{ab}\alpha^{\kappa+1}
   \end{equation} 
   with $C_{ab}=\max(\norm{a}_\infty,\norm{b}_\infty)$.
   Consequently, with $L^1:=L^1(\Omega_U)$ we get 
   \begin{equation}
   \begin{aligned}   
       \sigma \alpha \norm{u-\uopt_\alpha}_{L^1}-C\alpha^{\kappa+1}
       &\stackrel{\eqref{E:struct2}}{\le}
       \sigma \alpha \norm{u-\uopt_\alpha}_{L^1}
           - \sigma \alpha \norm{u-\uopt_\alpha}_{L^1(I_\alpha)}\\
       &\stackrel{\phantom{\eqref{E:struct2}}}{=}
         \sigma \alpha \norm{u-\uopt_\alpha}_{L^1(I_\alpha^c)}\\
       &\stackrel{\eqref{E:skpest1}}{\le}
        (\dual B\popt_\alpha,u-\uopt_\alpha)_{I_\alpha^c}\\
       &\stackrel{\phantom{\eqref{E:struct2}}}{=}
        (\dual B\popt_\alpha,u-\uopt_\alpha)
           -(\dual B\popt_\alpha,u-\uopt_\alpha)_{I_\alpha}\\
       &\stackrel{\eqref{E:skpest2}}{\le}
        (\dual B\popt_\alpha,u-\uopt_\alpha) + C\alpha^{\kappa+1}.
   \end{aligned} 
   \end{equation} 
    Rearranging terms, we conclude
   \begin{equation}
       \sigma \alpha \norm{u-\uopt_\alpha}_{L^1}
       \le (\dual B\popt_\alpha,u-\uopt_\alpha)  + C\alpha^{\kappa+1}.
   \end{equation} 
   Taking $u:=\uopt_0$ in the previous equation and adding 
   the necessary condition \eqref{VarIneq} for $\uopt_0$ 
   for the special case $u:=\uopt_\alpha$, i.e.,
   \begin{equation}
        (-\dual B\popt_0,\uopt_0 - \uopt_\alpha ) \ge 0,
   \end{equation} 
   we get the estimate
   \begin{equation}
   \begin{aligned}
      \sigma \alpha \norm{\uopt_0-\uopt_\alpha}_{L^1}
      &\le 
       (\dual B(\popt_\alpha-\popt_0),\uopt_0-\uopt_\alpha) 
              + C\alpha^{\kappa+1}\\
      &=
         - \norm{\yopt_{\alpha}-\yopt_0}_I^2 
              + C\alpha^{\kappa+1},
   \end{aligned}
   \end{equation} 
   from which the claim follows.

   The improved estimate can be established as in the proof of 
   Theorem~\ref{T:regmain}.
\end{proof}

The $\popt_\alpha$-measure condition \eqref{E:struct2} 
is slightly stronger than what
actually is necessary in order to obtain the above convergence rates
in the control. 
\begin{coro}
  Let $\uopt_0$ be a solution of \eqref{OCPl} and
  let us assume that 
  the $\sigma$-condition \eqref{E:sigmacond} is valid.
  
  If the convergence rate 
  $ \norm{\uopt_\alpha - \uopt_0}_{L^p(\Omega_U)}^p 
       \le C\alpha^\kappa 
  $
  is known to hold for some real $p\ge 1$ and some real $\kappa > 0$,
  then the measure condition 
          \begin{equation}
               \meas(\twoset{x\in \Omega_U}{
             \alpha (a+\epsilon)
                 \le   -\dual B\popt_\alpha(x) \le
             \alpha (b-\epsilon) }
                        \le \frac{C}{\epsilon^p} \alpha^\kappa
          \end{equation}
  is fulfilled for each $0<\epsilon<\sigma$.
\end{coro}
\begin{proof}
    This follows from the proof of Theorem~\ref{T:necmeasure}.
\end{proof}

If the limit problem is of certain regularity, the 
$\popt_\alpha$-measure condition
is not stronger than the $\popt_0$-measure condition, and, as
we show afterwards, both conditions coincide. 

\begin{lemm}\label{L:palphamcond}
 Let Assumption~\ref{A:sourcestruct} hold with $\meas(A^c)=0$
 ($\popt_0$-measure condition is valid a.e. on $\Omega_U$).
 Let furthermore $\kappa \ge 1$ and \eqref{E:dualTlinfty} be valid.
 Then the $\popt_\alpha$-measure condition \eqref{E:struct2}
 is fulfilled.
\end{lemm} 
\begin{proof}
Since the set $I_\alpha$ from \eqref{E:setialpha} fulfills
$
    I_\alpha
    \subset \twoset{x\in \Omega_U}{\abs{\dual B\popt_\alpha(x)}
          \le C\alpha}
$
with $C=\max(\norm{a}_\infty,\norm{b}_\infty)$, we conclude with
\eqref{E:dualTlinfty} and Theorem~\ref{T:regmain} that if 
$x\in I_\alpha$ and $\kappa \ge 1$, we have
\[
   \abs{\dual B\popt_0(x)}\le  \abs{\dual B\popt_\alpha(x)} 
     +  \abs{\dual B(\popt_0-\popt_\alpha)(x)} 
   \le C\alpha.
\]
With the $\popt_0$-measure condition \eqref{E:struct}
we obtain the estimate
\[
   \meas(I_\alpha) \le
   \meas(\twoset{x\in \Omega_U}{\abs{\dual B\popt_0(x)} \le C\alpha})
      \le C\alpha^\kappa,
\]
which concludes the proof.
\end{proof}

\begin{coro}
   Let a bang-bang solution be given which fulfills \eqref{E:bangbangsol}. 
   In the case of $\kappa > 1$,
   \eqref{E:dualTlinfty}, and the $\sigma$-condition \eqref{E:sigmacond}, 
   both measure conditions are equivalent.
\end{coro}
\begin{proof}
   One direction of the claim, namely 
   ``$\popt_0$-m.c. $\Rightarrow$ $\popt_\alpha$-m.c.'',
   has already been shown in Lemma~\ref{L:palphamcond}.

   For the other direction, we know from Theorem~\ref{T:convpalpha}
   that the convergence rates \eqref{E:convpalpharates} and
   \eqref{E:convpalpharateimp} hold,
   which by \eqref{E:dualTlinfty} and Theorem~\ref{T:necmeasure} imply 
   the $\popt_0$-measure condition.
\end{proof}

Let us now consider the situation that the optimal adjoint state
fulfills the regularity
\begin{equation}\label{E:hreg}
  \exists\ C>0 : 
   \norm{\partial_x \dual B\popt_\alpha}_{L^\infty(\Omega_U)} \le C
\end{equation}
with a constant $C>0$ independent of $\alpha$ 
and $\partial_x$ denoting the weak differential operator.
This bound is valid, e.g.,
for Example~\ref{exam:timedep}.2 (located controls), 
see \cite[p. 30]{dissnvd} for a proof.

Furthermore, we assume the Sobolev regularity
$a$, $b\in W^{1,\infty}(\Omega_U)$ for the control bounds.

Since the orthogonal projection possesses 
for $f\in W^{1,\infty}(\Omega_U)$
the property
\[
 \norm{\partial_x P_{\Uad}(f)}_{L^\infty(\Omega_U)} \le
    \norm{\partial_x f}_{L^\infty(\Omega_U)} 
      + \norm{\partial_x a}_{L^\infty(\Omega_U)}
      + \norm{\partial_x b}_{L^\infty(\Omega_U)},
\]
see, e.g. \cite[Corollary 2.1.8]{ziemer},
we obtain with the projection formula \eqref{FONC} and
the constant
\begin{equation}\label{E:Cab}
     C_{ab} := \norm{\partial_x a}_{L^\infty(\Omega_U)}
      + \norm{\partial_x b}_{L^\infty(\Omega_U)}
\end{equation}
a bound on the derivative of the optimal control, namely
\begin{equation}\label{E:uaainf}
   \norm{\partial_x \uopt_\alpha}_{L^\infty(\Omega_U)}
   \le \frac 1\alpha 
        \norm{\partial_x \dual B\popt_\alpha}_{L^\infty(\Omega_U)}  
        + C_{ab}
  \stackrel{\eqref{E:hreg}}{\le} C\frac 1\alpha,
\end{equation}
if $\alpha > 0$ is sufficiently small. 

If the $\popt_\alpha$-measure condition \eqref{E:struct2} is valid,
this decay of smoothness in dependence of
$\alpha$ can be relaxed in weaker norms, as the following Lemma shows.

\begin{lemm}[Smoothness decay in the derivative]\label{L:smdecderiv}
   Let the  $\popt_\alpha$-measure condition \eqref{E:struct2}
   be fulfilled and 
   the regularity condition \eqref{E:hreg} be valid
   as well as $a$, $b\in W^{1,\infty}(\Omega_U)$.
   Then there holds with the constant $C_{ab}$ defined in \eqref{E:Cab}
   for sufficiently small $\alpha >0$ and each $p$ with $1\le p < \infty$
   the inequality
   \begin{equation}
    \norm{\partial_x \uopt_\alpha}_{L^p(\Omega_U)}
     \le C\max(C_{ab},\alpha^{\kappa/p-1})
    \end{equation}
    with a constant $C>0$ independent of $\alpha$.

    Note that $C_{ab}=0$ in the case of constant control bounds 
    $a$ and $b$.
\end{lemm}
\begin{proof}
   We invoke \eqref{E:struct2} and \eqref{E:uaainf} to get 
   the estimate
   \begin{equation*}
   \begin{aligned}
       \norm{\partial_x \uopt_\alpha}_{L^p(\Omega_U)}^p
       &\le  \meas(I_\alpha) \norm{\partial_x \uopt_\alpha}
                     _{L^\infty(\Omega_U)}^p
               + \meas(\Omega_U)C_{ab}^p\\
       &\le C\max(\alpha^{\kappa-p},C_{ab}^p) 
   \end{aligned}
   \end{equation*}
   with the set $I_\alpha$ from \eqref{E:setialpha}. 
\end{proof}

Let us now briefly sketch an application of the previous lemma 
in the numerical analysis of a suitable finite element discretization of
Example~\ref{exam:timedep}.2 (located controls), 
which has been analyzed in detail recently in
\cite{dissnvd}, see also \cite{danielshinze}, 
founded on a novel discretization
scheme proposed in \cite{DanielsHinzeVierling}. 

Discretizing the regularized problem~\eqref{OCP} in space and time,
one ends up with a problem 
$(\mathbb P_{kh})$
depending on the regularization parameter $\alpha > 0$ and the grid
sizes $k$ and $h$ for the time and space grid, respectively, related
to the finite element discretization. 
This discretization is used later in the numerics section, where it
is described in more detail. 

One can show that this problem has again a unique solution 
$\uopt_{\alpha,kh}$
and that the error fulfills 
($\uopt_0$ the unique solution of \eqref{OCPl})
  \begin{multline}\label{E:theorem77}
       \norm{\uopt_0-\uopt_{\alpha,kh}}_U^2 
     + \norm{\uopt_0-\uopt_{\alpha,kh}}_{L^1(\Omega_U)} \\
       \le C \left( \alpha + h^2 + k^2
      \max(1,C_{ab},\alpha^{\kappa/2-1}) 
      \right)^\kappa 
  \end{multline}
with $C>0$ independent of $\alpha$, $k$, and $h$, 
see \cite[Theorem 77]{dissnvd}. 

Here, Lemma~\ref{L:smdecderiv} is used in the proof to get the
factor $\alpha^{\kappa/2-1}$. This factor is obviously better (if
$\alpha \to 0$) than $\alpha^{-1}$ which one would get using 
estimate \eqref{E:uaainf} only.

Using Lemma~\ref{L:smdecderiv} one can also show error estimates for
other quantities, e.g., an adjoint state error. 

\section{A numerical example}

To validate numerically the new convergence rates for the 
regularization error given in Theorem~\ref{T:regmain}.3,
we construct in the next subsection
a known (unique) solution $\uopt_0$ together with its problem data.

The problem data (but not the solution $\uopt_0$)
is used to numerically solve the \emph{regularized} problem \eqref{OCP}.
We describe in the next but one subsection how this approximation
is computed.

In the last subsection, we present and discuss the numerical results.

\subsection{A limit problem with given unique solution}

To build a concrete instance of the limit problem \eqref{OCPl},
we consider the situation of Example~\ref{exam:timedep}.2 
(heat equation with located controls). 

We modify the heat equation \eqref{e:heateq} 
in that we take into account
a fixed initial value $y(0)=y_0$, which 
can be interpreted as a modification of $z$ in \eqref{OCPl},
more precisely $z = y_d - S(0,y_0)$,
where $S(f,g)$ denotes the solution of the heat equation
\begin{equation}\label{E:heatfg}
\begin{aligned} \partial_t y -\Delta y &= f &&\text{in }I\times\Omega\,,\\
y&=0&&\text{in } I\times\partial\Omega\,,\\
y(0)&=g&&\text{in } \Omega\,.
\end{aligned}
\end{equation}

We consider a given exact solution of the limit 
problem \eqref{OCPl} which we denote by $(\uopt,\yopt,\popt)$, thus
omitting the index for $\alpha = 0$. Please take care of the fact that
$\yopt = S(B\uopt,y_0)$ in what follows, which is not $T\uopt$.
In the numerical procedure described below, we of course only
make use of the problem data. The solution triple $(\uopt,\yopt,\popt)$
is only used to evaluate the error norms.

To understand the construction of the test example,
let us elaborate in more detail the weak formulations of the solution
operator $S$ and its adjoint for Example~\ref{exam:timedep}.2. 
It also motivates the discretization schemes stated below.

With the space
\[ 
    W(I):=\twoset{v\in \LIIV}{v_t \in \LIIVd} ,
\]
the operator 
\begin{equation}\label{E:operatorS}
   S: \LIIVd\times \H \rightarrow W(I),\quad (f,g) \mapsto y:= S(f,g) ,
\end{equation}
denotes the weak solution operator associated with the heat equation
\eqref{E:heatfg}, which is defined as follows. 
For $(f,g) \in \LIIVd\times \H$ the function $y\in W(I)$ with
$\langle \cdot,\cdot \rangle := \langle \cdot,\cdot \rangle_{\Vd\V}$
satisfies the two equations
\begin{subequations}\label{E:WF}
\begin{align}
 y(0)={}&g \\
\begin{split}
\int_0^T \bigg\langle \partial_t y(t),v(t)\bigg\rangle
   + a(y(t),v(t))\, dt
={}& \int_0^T \bigg\langle f(t),v(t)\bigg\rangle\, dt\\
    \phantom{=}{}&\quad\forall\, v\in \LIIV.
\end{split}
\end{align}
\end{subequations}
Note that by the embedding 
$W(I)\hookrightarrow C([0,T],\H)$, see, e.g., \cite[Theorem 5.9.3]{evans},
the first relation is meaningful.\\
In the preceding equation, the bilinear form 
$a:H^1(\Omega)\times H^1(\Omega)\to\mathbb R$ is given by
\[
a(f,g):= \int_\Omega \nabla f(x) \nabla g(x)\ dx.
\]
The equations \eqref{E:WF} yield an operator $S$
in the sense of \eqref{E:operatorS}:
\begin{lemm}[Properties of the solution operator $S$]
\mbox{}
\begin{enumerate}
\item For every 
$(f,g) \in \LIIVd\times \H$ a unique state $y \in W(I)$ 
satisfying \eqref{E:WF} exists. Thus the operator $S$ from
\eqref{E:operatorS} exists. Furthermore the state fulfills
\begin{equation}\label{E:stabS}
\norm{y}_{W(I)} \le C 
    \left(\norm{f}_{\LIIVd}+\norm{g}_{\H}\right).
\end{equation}
\item Consider the bilinear form $A:W(I)\times W(I)\to\mathbb R$
given by
\begin{equation}\label{bilinA}
A(y,v):= \int_0^T -\bigg\langle v_t,y\bigg\rangle + 
 a(y,v)\, dt + \bigg\langle y(T),v(T)\bigg\rangle
\end{equation}
with $\langle \cdot,\cdot \rangle := \langle \cdot,\cdot \rangle_{\Vd\V}$.
Then for $y\in W(I)$, equation \eqref{E:WF} is equivalent to
\begin{equation}\label{E:WFM}
A(y,v) = \int_0^T \bigg\langle f,v\bigg\rangle\, dt +
    (g,v(0))_{\H}\quad\forall\ v\in W(I).
\end{equation}
Furthermore, $y$ is the only function in $W(I)$ fulfilling
equation \eqref{E:WFM}.
\end{enumerate}
\end{lemm}
\begin{proof}
This can be derived using standard results, see \cite[Lemma~1]{dissnvd}.
\end{proof}

One can show with standard results, see, e.g.,
\cite[Lemma~2]{dissnvd}, that the optimal adjoint state 
$\popt_\alpha \in W(I)$ is the unique weak solution 
defined and uniquely determined by the adjoint equation
\begin{equation}\label{E:AA}
A(v,\popt)
=
\int_0^T \langle \yopt-y_d,v\rangle_{\Vd\V}\, dt 
\quad \forall\ v\in W(I).
\end{equation}
This equation corresponds to the backward heat equation
\begin{equation}\label{E:bwheat}
 \begin{aligned} -\partial_t \bar p -\Delta \bar p &= \yopt-y_d 
         &&\text{in }I\times\Omega\,,\\
\bar p&=0&&\text{on } I\times\partial\Omega\,,\\
\bar p(T)&=0 &&\text{on } \Omega.
\end{aligned}
\end{equation}

Let us now construct the test example.

We make use of the fact that instead of the linear control
operator $B$, given by \eqref{E:Bloccont},
we can also use an \emph{affine linear}
control operator
\begin{equation}\label{E:Btilde}
\tilde B: U\rightarrow L^2(I,\Vd)\,,\quad u\mapsto g_0 + Bu
\end{equation}
where $g_0$ is a fixed function of certain regularity
since $g_0$ can be interpreted as a modification of $z$.

With a space-time domain $\Omega\times I := (0,1)^2 \times (0,0.5)$,
thus $T_e=0.5$,
we choose the optimal control to be the lower bound of the 
admissible set, i.e., $\uopt := a_1 := -0.2$.
For the upper bound we set $b_1:= 0.2$.

With the function
\[
      g_1(x_1,x_2) := \sin(\pi x_1)\sin(\pi x_2)
\]
the optimal adjoint state is chosen as
\[
    \popt(t,x_1,x_2) := (T_e-t)^{1/\kappa} g_1(x_1,x_2)  
\]
for some fixed $\kappa >0$ specified below.

With the constant $a:=2$, we take for the optimal state
\begin{equation*}
      \yopt(t,x_1,x_2) 
          := \cos\left(\frac{t}{T_e}\,2\pi a\right)\cdot g_1(x_1,x_2)\,,
\end{equation*}
from which we derive by \eqref{E:bwheat}
\[
    -\partial_t \popt - \Delta \popt
       = \frac 1\kappa (T_e-t)^{1/\kappa - 1} g_1
           - (T_e-t)^{1/\kappa}\Delta g_1 = \yopt - y_d \,,
\]
which gives $y_d$.

We also get the initial value of the optimal state $\yopt$: 
\[
      y_0(x_1,x_2) = \yopt(0,x_1,x_2) = g_1(x_1,x_2)\,.
\]
Finally we obtain
\begin{equation*}
\begin{aligned}
      g_0&= \partial_t \yopt - \Delta \yopt  -  B\uopt\\
         &= g_1 2 \pi \left( -\frac{a}{T_e}
              \sin\left(\frac{t}{T_e}\,2\pi a\right)
          + \pi\cos\left(\frac{t}{T_e}\,2\pi a\right) \right)
                    -g_1 \cdot\uopt\,.
\end{aligned}
\end{equation*}

This example fulfills the measure condition \eqref{E:struct} of
Assumption~\ref{A:sourcestruct} with $\meas(A^c) = 0$ and
exponent $\kappa$ from the definition of $\popt$.

\subsection{Discretization of the regularized problem}

We now describe the discretized regularized optimal control problem
$(\mathbb P_{kh})$ which is solved as an approximation for
\eqref{OCP}.

Consider a partition $0=t_0 <  t_1 < \dots < t_M=T_e$ of the time 
interval $\bar I=[0,T_e]$. With $I_m=[t_{m-1},t_m)$ we have 
$[0,T_e)=\bigcup_{m=1}^M I_m$.
Furthermore, let $t_m^*=\frac{t_{m-1}+t_m}{2}$ for $m=1,\dots,M$
denote the interval midpoints. By 
$0=:t_0^* < t_1^* < \dots < t_M^* < t_{M+1}^*:=T_e$ we get a 
second partition of $\bar I$, the so-called \emph{dual partition}, namely 
$[0,T_e)=\bigcup_{m=1}^{M+1} I_m^*$, with $I_m^*=[t_{m-1}^*,t_m^*)$.
The grid width of the first mentioned (primal) partition is given by the 
parameters $k_m=t_m-t_{m-1}$ and $k=\max_{1\le m\le M} k_m$.
Here and in what follows we assume $k<1$.
We also denote by $k$ (in a slight abuse of notation) the grid itself.

On these partitions of the time interval, we define the Ansatz and
test spaces of the Petrov--Galerkin schemes. 
These schemes will replace the continuous-in-time weak 
formulations of the state equation and the adjoint
equation, i.e., \eqref{E:WFM} and \eqref{E:AA}, respectively.
To this end, we define at first for an arbitrary Banach space $X$
the semidiscrete function spaces
\begin{subequations}\label{e:semfs}
\begin{align}
P_k(X):&=\twoset{v\in C([0,T],X)}{\restr{v}{I_m}\in \mathcal P_1(I_m,X)}
\hookrightarrow H^1(I,X),\\  
P_k^*(X):&=\twoset{v\in C([0,T],X)}{\restr{v}{I_m^*}\in 
\mathcal P_1(I_m^*,X)}
\hookrightarrow H^1(I,X),\\  
\shortintertext{and}
Y_k(X):&=\twoset{v:[0,T]\rightarrow \dual{X}}{\restr{v}{I_m}\in 
\mathcal P_0(I_m,X)}\,.  
\end{align}
\end{subequations}
Here, $\mathcal P_i(J,X)$, $J\subset \bar I$, $i\in\{0,1\}$, is
the set of polynomial functions in time with degree of at most $i$ on the
interval $J$ with values in $X$.

Note that we can extend the bilinear form $A$ of
\eqref{bilinA} in its first argument to $W(I)\cup Y_k(\V)$, thus
consider the operator
\begin{equation}\label{bilinAe}
A:W(I)\cup Y_k(\V) \times W(I) \rightarrow \mathbb{R},\quad
    \text{$A$ given by \eqref{bilinA}}. 
\end{equation}

Using continuous piecewise linear functions in space, we can
formulate fully discretized variants of the state and adjoint
equation.

We consider a regular triangulation $\mathcal T_h$ of $\Omega$
with mesh size $h:=\max_{T\in\mathcal T_h}\allowbreak\diam(T)$,
see, e.g., \cite[Definition (4.4.13)]{brenner-scott}, and
$N=N(h)$ triangles.
We assume that $h < 1$.
We also denote by $h$ (in a slight abuse of notation) the grid itself.

With the space
\begin{equation}
    X_h := \twoset{\phi_h\in C^0(\bar\Omega)}{\restr{\phi_h}{T}\in
                 \mathcal P_1(T,\mathbb R)
                     \quad\forall\ T\in\mathcal T_h}
\end{equation}
we define $X_{h0} := X_h \cap \V$ to discretize $\V$. 

We fix fully discrete ansatz and test spaces,
derived from their semidiscrete counterparts from \eqref{e:semfs},
namely
\begin{equation}\label{E:spaceskh}
    P_{kh}:=P_k(X_{h0}),\quad P_{kh}^*:=P_{kh}^*(X_{h0}),\quad 
   \text{and}\ Y_{kh}:=Y_k(X_{h0}).
\end{equation}
With these spaces, we introduce 
fully discrete state and adjoint equations as follows.
\begin{defi}[Fully discrete adjoint equation]\label{D:pkh}
For $h\in \LIIVd$ find $p_{kh}\in P_{kh}$ such that
\begin{equation}\label{E:AdjDiscrh}
A(\tilde y,p_{kh})
=\int_0^T \langle h(t),\tilde y(t)\rangle_{\Vd\V}\, dt
\quad\forall\ \tilde y \in Y_{kh}.
\end{equation}
\end{defi}
\begin{defi}[Fully discrete state equation]\label{D:ykh}
For $(f,g)\in L^2(I,\Vd)\times \H$ find $y_{kh}\in Y_{kh}$, such that
\begin{equation}\label{E:WFDh}
A(y_{kh},v_{kh})= \int_0^T \langle f(t),v_{kh}(t)\rangle_{\Vd\V}\, dt +
    (g,v_{kh}(0)) \quad\forall\ v_{kh}\in P_{kh}.
\end{equation}
\end{defi}

Existence and uniqueness of these two schemes follow as in the 
semidiscrete case discussed in \cite{DanielsHinzeVierling} or
\cite[section~2.1.2]{dissnvd}.
For error estimates of the two schemes, we refer again to
\cite{dissnvd} or \cite{danielshinze}.

We are now able to introduce
the discretized optimal control problem which reads

\begin{equation}\label{OCPkh}\tag{$\mathbb P_{kh}$}
\begin{aligned}
&\min_{ y_{kh}\in Y_{kh},u\in\Uad} 
     J(y_{kh},u)=\min \frac{1}{2}\norm{y_{kh}-y_d}^2_I+
             \frac{\alpha}{2}\norm{u}^2_U,\\
&\text{s.t. } y_{kh}=S_{kh}(Bu,y_0)
\end{aligned}
\end{equation}
where $\alpha$, $B$, $y_0$, $y_d$, and $\Uad$ are chosen as for
\eqref{OCP} and
$S_{kh}$ is the solution operator associated to the fully 
discrete state equation \eqref{E:WFDh}. 
Recall that the space $Y_{kh}$ was introduced in \eqref{E:spaceskh}.

For every $\alpha > 0$, this problem admits a unique solution triple
($\uoptd$, $\yoptd$, $\poptd$) where
$\yoptd = S_{kh}(B\uoptd,y_0)$ and 
$\poptd$ denotes the discrete adjoint state which is
the solution of the fully discrete adjoint equation 
\eqref{E:AdjDiscrh} with right-hand side $h:=\yoptd-y_d$.
The first order necessary and sufficient optimality condition for
problem \eqref{OCPkh} is given by
\begin{equation}\label{VarIneqkh}
   \uoptd\in\Uad,\quad
   \left( \alpha\uoptd + \dual B\poptd,u-\uoptd\right)_U 
   \ge 0\quad \forall\ u\in\Uad,
\end{equation}
which can be rewritten as
\begin{equation}\label{FONCkh}
    \uoptd = P_{\Uad}\left(-\frac{1}{\alpha}\dual B\poptd\right).
\end{equation}
The above mentioned facts can be proven in the same way
as for the continuous problem \eqref{OCP}. 

Note that the control space $U$ is not discretized in the
formulation \eqref{OCPkh}. In the numerical treatment, 
the relation \eqref{FONCkh} is instead exploited to get a discrete 
control.
This approach is called
\emph{Variational Discretization} and was introduced 
in \cite{Hinze2005}, see also \cite[Chapter 3.2.5]{hpuu}
for further details.

\subsection{Numerical results}

We solve numerically the regularized discretized problem
\eqref{OCPkh} as an approximation of the limit problem \eqref{OCPl}
in the situation of Example~\ref{exam:timedep}.2 with data given in the
last but one subsection. Recall that we denote by
 $\uopt_{kh}$ the former, by $\uopt$ the latter.

We investigate the behavior of the error $\norm{\uopt_{kh}-\uopt}$
if $\alpha \to 0$
for fixed small discretization parameters $k$ and $h$ and different
values of the parameter $\kappa$ from the measure condition 
\eqref{E:struct}.

To solve $(\mathbb P_{kh})$, a 
fixed-point iteration on the equation \eqref{FONCkh} is performed:

Each fixed-point iteration is initialized with the starting value 
$u_{kh}^{(0)}:=a_1$ which is the lower bound of the admissible set.
As a stopping criterion for the fixed-point iteration, we require for
the discrete adjoint states belonging to the current
and the last iterate that
\[
    \norm{ \dual B \left(p_{kh}^{(i)} - p_{kh}^{(i-1)}\right)
         }_{L^\infty(\Omega\times I)} < t_0
\]
where $t_0:=10^{-5}$ is a prescribed threshold.

We end up with what we denote by $u_{kh}$ in the tables below:
An approximation of $\uoptd$.

The idea is that if $\alpha$ is not to small in comparison to 
$k$ and $h$, we expect to have
$\norm{u_{kh}-\uopt} \approx \norm{\uopt_\alpha - \uopt}$,
which means that the influence of the discretization is negligible
in relation to the influence of the regularization.  

Here, we report only on the errors in the optimal control since
we observed no or only poor convergence in the error of the optimal
state and adjoint state, respectively. This might be due to the fact
that the influence of the space- and time-discretization error is much
larger than that of the regularization error.
This phenomenon was also observed for elliptic problems, compare
\cite{wachsmuth1}. 

We consider a fixed fine space-time mesh with 
$\text{Nh}=(2^5+1)^2$ nodes in space and
$\text{Nk}=(2^{11}+1)$ nodes in time.
The regularization parameter $\alpha = 2^{-\ell}$ is considered for 
$\ell=1,2,3,4,5,6$. 

The problem is solved for different values of $\kappa$, namely
$\kappa = 0.3$, $0.5$, $1$, and $2$. 

Let us remark that the convergence of the fixed-point iteration
for our example does not depend on the starting value.

As one can see from the \emph{experimental order of convergence} (EOC) 
in the Tables~\ref{tab:ex2k03}, \ref{tab:ex2k05},
\ref{tab:ex2k1}, and \ref{tab:ex2k2}, 
the new convergence rates of Theorem~\ref{T:regmain}.3, more
precisely \eqref{E:regconvumeas} and \eqref{E:regconvumeas2},
can be observed numerically.
It seems that they cannot be improved
any further. In Figure~\ref{fig:ex2k1}, the convergence of
the computed optimal control to the limit control is depicted if
$\alpha \to 0$.

\begin{table}[hbt]
\begin{center}
\begin{tabular}{ccccccc}
\hline
& $\norm{\uopt-u_{kh}}$
& $\norm{\uopt-u_{kh}}$
& $\text{EOC}$
& $\text{EOC}$\\
$\ell$ 
& $L^1(I,\mathbb R)$ 
& $L^2(I,\mathbb R)$ 
& $L^1$
& $L^2$\\
\hline
 1 & 0.09417668 & 0.13354708 &    /  &   / \\ 
 2 & 0.08837777 & 0.12648809 &  0.09 & 0.08\\ 
 3 & 0.07681662 & 0.11533688 &  0.20 & 0.13\\ 
 4 & 0.06212895 & 0.10353644 &  0.31 & 0.16\\ 
 5 & 0.05008158 & 0.09264117 &  0.31 & 0.16\\ 
 6 & 0.04011694 & 0.08237596 &  0.32 & 0.17\\ 
\hline
\end{tabular}
\end{center}
\caption{
     Errors and EOC in the control 
        ($\kappa = 0.3$, $\alpha \to 0$, $h$, $k$ fixed).}
\label{tab:ex2k03}
\end{table}

\begin{table}[hbt]
\begin{center}
\begin{tabular}{ccccccc}
\hline
& $\norm{\uopt-u_{kh}}$
& $\norm{\uopt-u_{kh}}$
& $\text{EOC}$
& $\text{EOC}$\\
$\ell$ 
& $L^1(I,\mathbb R)$ 
& $L^2(I,\mathbb R)$ 
& $L^1$
& $L^2$\\
\hline
 1 & 0.07912861 & 0.11494852 &    /  &   /  \\ 
 2 & 0.05957289 & 0.09753159 &  0.41 & 0.24 \\ 
 3 & 0.04204449 & 0.08187630 &  0.50 & 0.25 \\ 
 4 & 0.02963509 & 0.06865675 &  0.50 & 0.25 \\ 
 5 & 0.02084162 & 0.05749818 &  0.51 & 0.26 \\ 
 6 & 0.01463170 & 0.04811089 &  0.51 & 0.26 \\ 
\hline
\end{tabular}
\end{center}
\caption{
     Errors and EOC in the control 
        ($\kappa = 0.5$, $\alpha \to 0$, $h$, $k$ fixed).}
\label{tab:ex2k05}
\end{table}

\begin{table}[hbt]
\begin{center}
\begin{tabular}{ccccccc}
\hline
& $\norm{\uopt-u_{kh}}$
& $\norm{\uopt-u_{kh}}$
& $\text{EOC}$
& $\text{EOC}$\\
$\ell$ 
& $L^1(I,\mathbb R)$ 
& $L^2(I,\mathbb R)$ 
& $L^1$
& $L^2$\\
\hline
 1 & 0.04006495 & 0.07304858  &   /  &   /  \\ 
 2 & 0.02000722 & 0.05160925  & 1.00 & 0.50 \\ 
 3 & 0.00998774 & 0.03646496  & 1.00 & 0.50 \\ 
 4 & 0.00498724 & 0.02576440  & 1.00 & 0.50 \\ 
 5 & 0.00249053 & 0.01820019  & 1.00 & 0.50 \\ 
 6 & 0.00123906 & 0.01282180  & 1.01 & 0.51 \\ 
\hline
\end{tabular}
\end{center}
\caption{
     Errors and EOC in the control 
        ($\kappa = 1$, $\alpha \to 0$, $h$, $k$ fixed).}
\label{tab:ex2k1}
\end{table}

\begin{figure}
 \centering  
 \subfloat[$\ell=1$]{\includegraphics[trim=25mm 75mm 15mm 91mm,
        clip,width=0.3\textwidth]{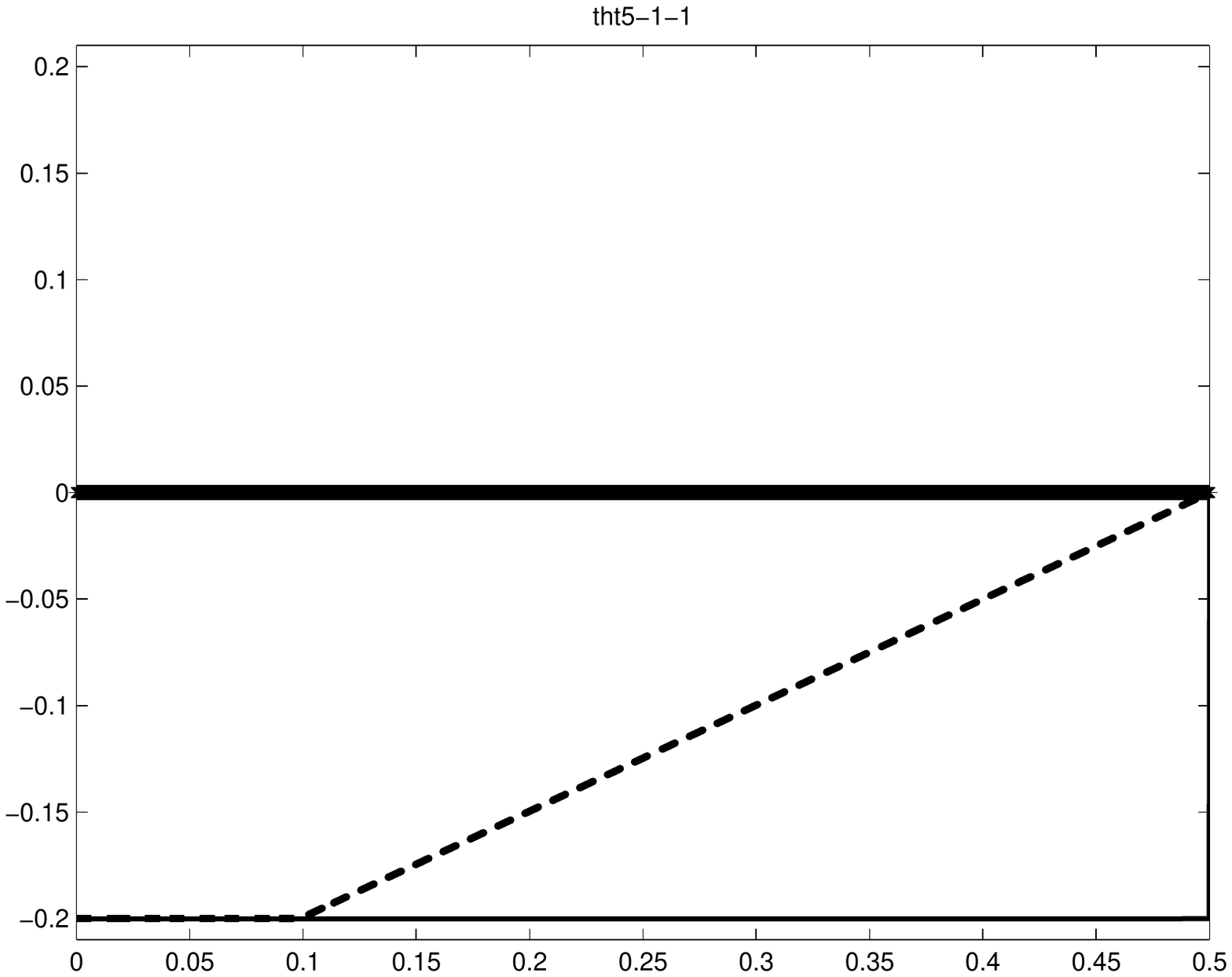}}
 \subfloat[$\ell=2$]{ \includegraphics[trim=25mm 75mm 15mm 91mm,
        clip,width=0.3\textwidth]{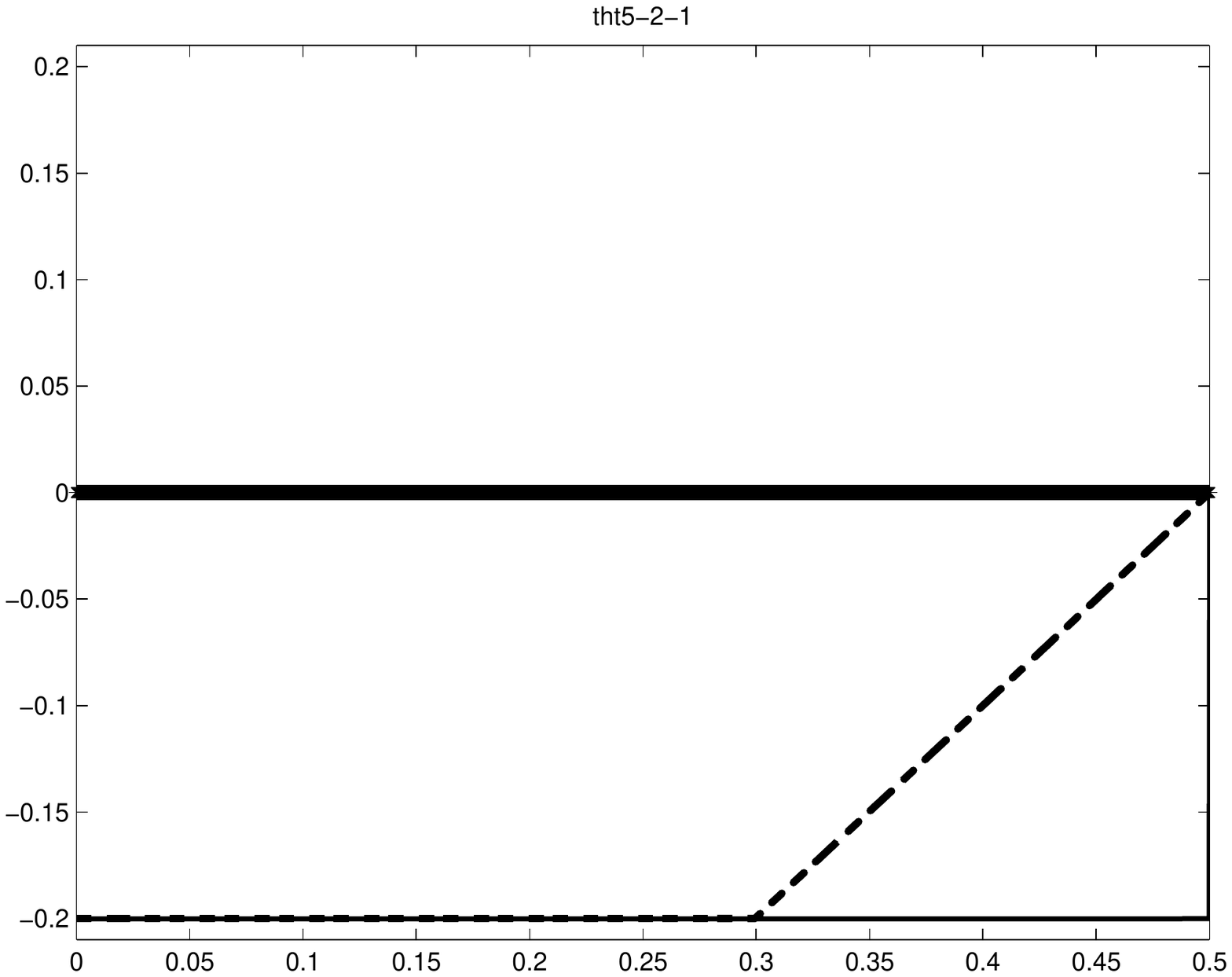}}
 \subfloat[$\ell=3$]{ \includegraphics[trim=25mm 75mm 15mm 91mm,
        clip,width=0.3\textwidth]{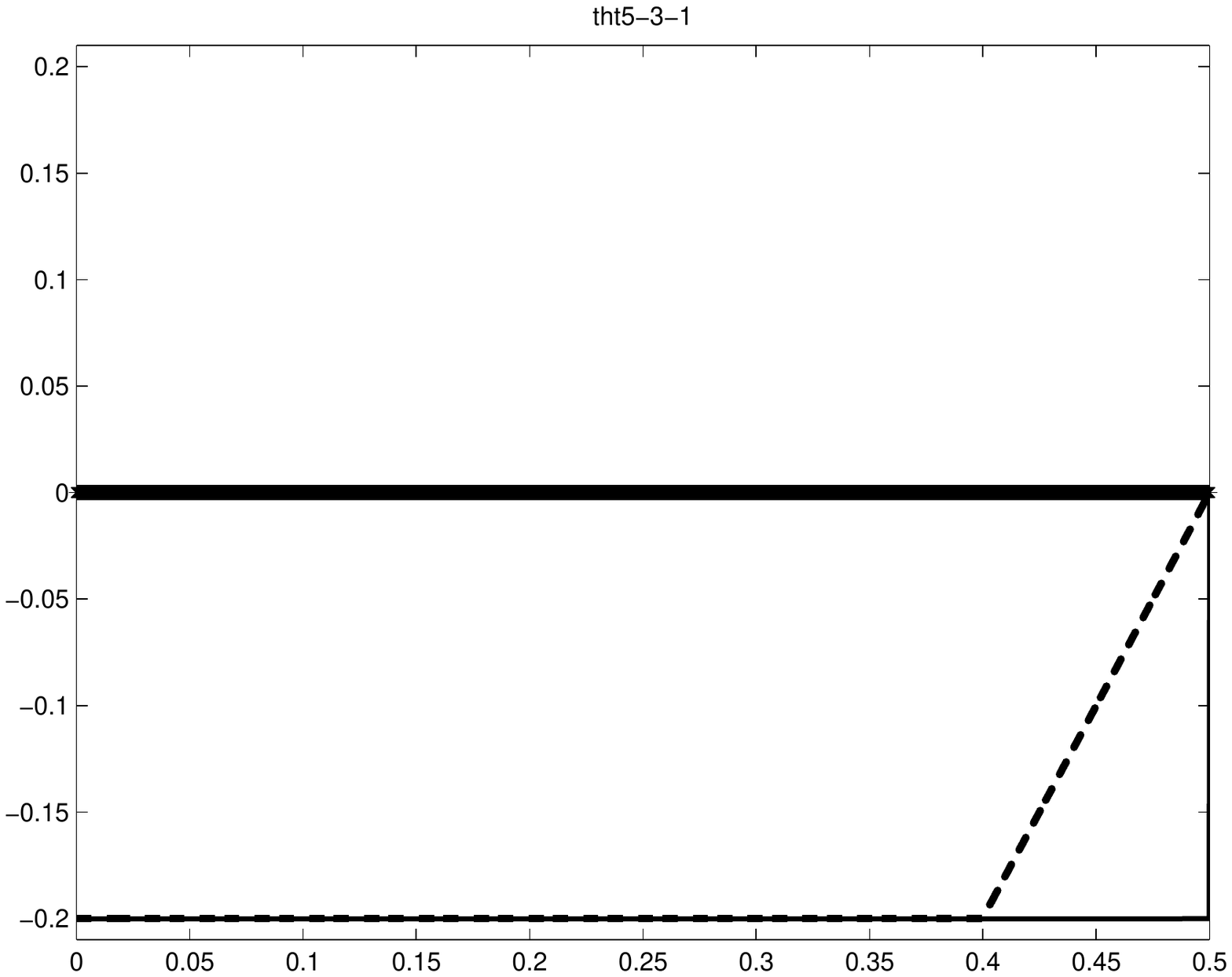}}
    \caption{
         Optimal control $\uopt$ (solid) and computed counterpart
       $u_{kh}$ (dashed) over time after level $\ell$
        ($\kappa = 1$, $\alpha \to 0$, $h$, $k$ fixed).
                }
    \label{fig:ex2k1}
\end{figure}

\begin{table}[hbt]
\begin{center}
\begin{tabular}{ccccccc}
\hline
& $\norm{\uopt-u_{kh}}$
& $\norm{\uopt-u_{kh}}$
& $\text{EOC}$
& $\text{EOC}$\\
$\ell$ 
& $L^1(I,\mathbb R)$ 
& $L^2(I,\mathbb R)$ 
& $L^1$
& $L^2$\\
\hline
 1 & 0.01081546 & 0.03305084  &   /  &   /  \\ 
 2 & 0.00279478 & 0.01690248  & 1.95 & 0.97 \\ 
 3 & 0.00074507 & 0.00878066  & 1.91 & 0.94 \\ 
 4 & 0.00020543 & 0.00463711  & 1.86 & 0.92 \\ 
 5 & 0.00005823 & 0.00246523  & 1.82 & 0.91 \\ 
 6 & 0.00001564 & 0.00125068  & 1.90 & 0.98 \\  
\hline
\end{tabular}
\end{center}
\caption{
     Errors and EOC in the control 
        ($\kappa = 2$, $\alpha \to 0$, $h$, $k$ fixed).}
\label{tab:ex2k2}
\end{table}

\printbibliography
\end{document}